\documentclass[a4paper, 11pt, parskip=half]{amsart}

\usepackage{custom-template, shortcuts}
\usepackage{mathtools}
\usepackage{tabularx}
\usepackage{array}
\usepackage{cancel}
\usepackage{stmaryrd}
\begin{document}
	\title{Certain squarefree levels of reducible modular mod$\,\ell$ Galois representations}
	\date{}
	\author{Arvind Kumar and Prabhat Kumar Mishra}

	\address{Department of Mathematics, Indian Institute of Technology Jammu, Jagti, PO Nagrota, NH-44 Jammu 181221, J \& K, India\vspace*{-5pt}}
	\email{arvind.kumar@iitjammu.ac.in\vspace*{-6pt}}
	\email{2022rma0025@iitjammu.ac.in\vspace*{-6pt}}

	\subjclass[2020]{}
	
	\begin{abstract}
      %      Let $k\ge 2$ be an even integer and $\ell\ge \max\{5,k-1\}$ be a prime. We study the problem of determining the squarefree levels of newforms giving rise to the representation that is the direct sum of the trivial character and $(k-1)$st power of the $\rm mod\,\ell$ cyclotomic character. We determine the necessary and sufficient conditions for levels with two distinct prime factors to give rise to the above direct sum and we establish a congruence similar to  Ramanujan's $\mod{691}$ congruence. An application of our result shows that under certain assumptions on $k$ and $\ell$, there exist sets of primes with positive densities and a newform of level with factors from those explicit sets of primes satisfying the Ramanujan congruence. \\
Let $k \ge 2$ be an even integer, $ \ell \ge \max\{5, k-1\} $ be a prime, and $N$ be a squarefree positive integer. It is known that if the $\oq{mod\,\ell}$ Galois representation $\overline{\rho}_f$ associated with a newform $f$ of weight $k$, level $N$, and trivial nebentypus is reducible, then  $\overline{\rho}_f \simeq 1 \oplus \overline{\chi}_\ell^{k-1}$, up to semisimplification, where $\overline{\chi}_\ell^{}$ is the $\oq{mod\,\ell}$ cyclotomic character.  In this paper, we determine the necessary and sufficient conditions under which the $\oq{mod\,\ell}$ representation $1 \oplus \overline{\chi}_\ell^{k-1}$ arises from a newform of weight $k$, level $N$ with exactly two prime factors with specified Atkin-Lehner eigenvalues. Specifically, this proves a conjecture of Billerey and Menares when $N$ is a product of two primes under some mild assumption.
 As an application, we show that for any $\ell\ge 5$ and $k=2$ or $\ell+1$, there exist a large class of distinct primes $p$ and $q$ such that the $\oq{mod\,\ell}$ representation $1 \oplus \overline{\chi}_\ell^{k-1}$ arises from a newform of weight $k$ and level $pq$ with explicit Atkin-Lehner eigenvalues.
	\end{abstract}
	\maketitle

\section{Introduction}
Let $k\ge 2$ be an even integer and $\ell$ be a prime. Throughout the article, we assume that $N$ is a squarefree positive integer. We use the notation $\overline{K}$ to denote an algebraic closure of a number field $K$, $\bb{F}_{\ell}$ to denote the finite field with $\ell$ elements, and  $G_{\bb{Q}}$ to denote the absolute Galois group $\oq{Gal}(\overline{\bb{Q}}/\bb{Q})$.
%A continuous group homomorphism $\rho: G_{\bb{Q}}\rightarrow \oq{GL}_2(\overline{\bb{F}}_\ell)$ is called a $\oq{mod~\ell}$ Galois representation and such a representation is said to be odd if $\oq{det(\rho(\kappa)) = -1}$, where $\kappa\in G_{\bb{Q}}$ is the complex conjugation.
Let $\mathcal{S}_k(N)$ be the space of cusp forms of weight $k$  and level $N$, i.e., for the congruence subgroup $\Gamma_0(N)$.
%and $\mathcal{S}_k(N)^{\rm new}$ be its subspace of newforms. 
For a (normalized) newform $f(z) = \sum_{n\geq 1} a_f(n)e^{2\pi i n z}\in \mathcal{S}_k(N)$, from the works of Eichler-Shimura and Deligne, we have an odd semisimple $\oq{mod}\,\ell$ Galois representation 
\[
        \overline{{\rho}}_{f,\Lambda} : G_{\mathbb{Q}}\rightarrow \oq{GL}_2(\overline{\bb{F}}_{\ell})
\]
that is unique up to isomorphism and unramified at primes $q\nmid \ell N$ and satisfies 
$$\mathrm{tr}(\overline{\rho}_{f,\Lambda}(\mathrm{Frob}_q)) = a_f(q)\modlam ~~ \text{  and } ~~ \mathrm{det}(\overline{\rho}_{f,\Lambda}(\mathrm{Frob}_q)) = q^{k-1}\modlam,$$ 
where $\mathrm{Frob}_q \in\oq{Gal}(\overline{\mathbb{Q}}/\mathbb{Q})$ denotes a Frobenius element at $q$ and $\Lambda$ denotes a prime above $\ell$ in the coefficient field $\mathbb Q(a_f(n):n\ge 1)$ of $f$.

A $\oq{mod\,\ell}$ Galois representation $\overline{{\rho}}: G_{\bb{Q}}\rightarrow \oq{GL}_2(\overline{\bb{F}}_\ell)$ is said to be a modular $\oq{mod\,}\ell$ Galois representation if it arises from a newform (of trivial nebentypus), i.e., there exists a newform $f$ such that $\overline{\rho}\simeq \overline{\rho}_{f,\Lambda}$, for some prime ideal $\Lambda$ over $\ell$ in the coefficient field of $f$. It is natural to ask which odd $\oq{mod\,}\ell$ Galois representations are modular. Since a $\oq{mod\,}\ell$ Galois representation may not arise from a unique newform, it is equally important to investigate the level and weight of the newforms which give rise to $\bar{\rho}$.

\begin{comment}
For an odd, irreducible $\oq{mod}\,\ell$ Galois representation $\bar\rho$, it is known that it arises from a Hecke eigenform. 
 If $N(\bar\rho)$ is the \emph{prime-to-$\ell$} part of the Artin conductor of  $\bar\rho$ (see \cite[1.2]{ser}), then Caroyol \cite{car} and Livn\'{e} \cite{liv} showed that the conductor $N(\rho)$ divides the level of modular forms giving rise to $\bar\rho$. \textcolor{blue}{If $\bar{\rho}$ arises from a newform of level $N(\bar{\rho})$, we say it is modular of optimal level.} The work of Ribet \cite[Theorem 1.1]{rib1} 
 shows that the representation $\bar\rho$ arises from a Hecke eigenform of level $N(\bar\rho)$. Subsequently, Diamond and Taylor \cite{dl} determined all the positive integers $M> N(\bar\rho)$ such that there is a newform of level $M$ giving rise to $\bar\rho$ and {called these integers {\em `non-optimal levels'}.}
\end{comment}
    
 Khare and Wintenberger proved the Serre's conjecture which states that every odd irreducible $\oq{mod}\,\ell$ Galois representation $\overline{\rho}$ arises from an eigenform in the space $\mathcal{S}_k(N(\overline{\rho}))$, where $N(\overline \rho)$ is a positive integer coprime to $\ell$ and equals the Artin conductor of $\overline{\rho}$.  The integer $N(\overline{\rho})$ is called the \emph{optimal level}. Diamond and Taylor \cite{dl} studied the levels $M> N(\overline{\rho})$ of newform giving rise to $\overline{\rho}$ and {called these integers {\em `non-optimal levels'}.}

\pagebreak 

Suppose $\overline{\chi}_\ell$ represents the  $\oq{mod\,\ell}$ cyclotomic character of $G_{\bb{Q}}$. An  odd reducible modular $\oq{mod\,\ell}$ Galois representation of squarefree level is isomorphic to $1\oplus \overline{\chi}_{\ell}^{k-1}$, up to semisimplification, for some $k$ satisfying $\ell \ge k-1$ (for a proof, see \cite[Proposition 3.1]{bm}). 
{The Artin conductor of $\overline{\chi}_{\ell}$ is $1$, so the optimal level of $1\oplus \overline{\chi}_{\ell}^{k-1}$ is 1. In contrast to the irreducible case, the representation $1\oplus \overline{\chi}_{\ell}$ is not modular of optimal level for weight 2 despite being odd and semisimple.} 
However, for even weight $k\geq 4$ and prime $\ell>k+1$, Ribet \cite[Lemma 5.2]{rib2} has proved that the representation $1\oplus \overline{\chi}_\ell^{k-1}$ is modular of optimal level iff $\ell\mid \frac{B_k}{2k}$, where $B_k$ is the $k$th Bernoulli number. %Thus, we now need to consider the problem of determining the non-optimal levels of $1\oplus \overline{\chi}_\ell^{k-1}$.

This article focuses on studying the non-optimal squarefree levels of an odd, reducible modular $\oq{mod\,\ell}$ Galois representation, which is equivalent to studying the non-optimal squarefree levels of $1\oplus \overline{\chi}_{\ell}^{k-1}$, for even $k\ge 2$. 
Several works have been done to determine non-optimal prime levels of $1\oplus \overline{\chi}_{\ell}^{k-1}$ and we mention some of these results now. 
  For $k=2$, Mazur \cite[Proposition 5.12]{maz} first identified non-optimal prime levels of $1\oplus \overline{\chi}_\ell$  by showing that it arises from a weight $2$ newform of prime level $p$ iff $\ell\mid (p-1)/12$. 
 For $k\ge 4$ with $\ell>k+1$, Billerey-Menares \cite{bm} gave a necessary and sufficient conditions for a prime number to be a non-optimal level of $1\oplus \overline{\chi}_\ell^{k-1}$ (see also, \cite{df}). Gaba-Popa \cite{gp} and Kumar et al. \cite{kkms} refined these results by introducing the Atkin-Lehner eigenvalues of the newforms of the prime level involved.

\subsection{Conjecture for squarefree level}
Billerey-Menares \cite[Conjecture 3.2]{bm} proposed a conjecture  for determining all the non-optimal squarefree levels $N$ of the representation $1\oplus \overline{\chi}_\ell^{k-1}$, for any $k\ge 4$.  
In the following, we refine that conjecture by introducing an Atkin-Lehner eigensystem for $\Gamma_0(N)$, which is defined as a multiplicative function $\varepsilon: \mathcal{P}_N \rightarrow \{\pm 1\}$  with $\varepsilon(1)=1$, where $\mathcal{P}_{N}$ denotes the set of positive divisors of $N$. Note that $\varepsilon$ can also be considered as an Atkin-Lehner eigensystem for $\Gamma_0(M)$ for any $M\mid N$ by restricting it to $\mathcal{P}_M$, which we use throughout the article without mentioning explicitly. Let $\mathcal{M}_k(N)$ be the space of modular forms of weight $k$ and level $N$.  For a form $f\in \mathcal M_k(N)$, we say $\varepsilon$ is the Atkin-Lehner eigensystem of $f$  if $\varepsilon(p)$ is the eigenvalue of $f$ under the action of Atkin-Lehner operator $W_p$ for each $p\mid N$ and in this case we write $f\in \mathcal{M}_k^{(\varepsilon)}(N)$ and similarly we define $\mathcal{S}_k^{(\varepsilon)}(N)$.
%(see Section \ref{newform theory} for more details).

\begin{conjecture}[Generalized Billerey-Menares Conjecture]\label{conjecture}
Let $k\geq 4$,  $\ell>k+1$ be a prime, and $N=p_1\cdots p_t$, where $p_i's$ are distinct primes. Let $\varepsilon$ be an Atkin-Lehner eigensystem  for $\Gamma_0(N)$. Then following are equivalent:
\begin{enumerate}
 \item $ \ell \mid \frac{B_k}{2k}  \prod\limits_{i=1}^{t}(1+\varepsilon(p_i) p_i^{k/2}) \quad {\rm and} \quad \ell\mid (1+\varepsilon(p_i) p_i^{k/2})(1+\varepsilon(p_i) p_i^{k/2-1})~{{  for ~each} ~1\leq i\leq t}$.
    
   \item There exist  a newform $f\in \mathcal{S}_k^{(\varepsilon)}(N)$ and a prime ideal $\Lambda $ lying over $\ell$ in the coefficient field of $f$ such that
   \[
        \overline{\rho}_{f,\Lambda} \simeq 1\oplus \overline{\chi}_\ell^{k-1}. 
   \]
   \end{enumerate}
\end{conjecture}

%A  proof of the the reverse implication in Conjecture \ref{conjecture} is outlined in Section \ref{reverse implication}. 
{The reverse implication easily follows from the congruence of Lemma \ref{Converse}} by considering the constant terms and $p$th Fourier coefficients for each $p\mid N$.
In general, the above conjecture is not valid for $k=2$ and an example is given in Section \ref{examples} (see Example \ref{exm}). 
 In Theorem \ref{existence of eigenform}, we prove that the above conjecture is true for eigenforms (by which we mean eigenfunction of all the Hecke operators) instead of newforms, and its somewhat weaker version was proven in \cite[Theorem 3.5]{bm}.

We emphasize that Gaba-Popa \cite{gp} and Kumar et al. \cite{kkms} showed that Conjecture \ref{conjecture} is true if $N$ is a prime under some mild assumptions. Recently, Deo \cite[Corollary 1.8]{svd} has proved the original conjecture of Billerey-Menares \cite[Conjecture 3.2]{bm} in some cases under certain assumptions. The main aim of this article is to prove Conjecture \ref{conjecture} for levels with exactly two prime factors (see Corollary \ref{cor_conj}).

We remark that determining the non-optimal levels of the representation $1\oplus \overline{\chi}_\ell^{k-1}$ is equivalent to knowing the cases when a newform of weight $k$ and level $N$ is congruent to a suitable Eisenstein series of the same weight and level (in the spirit of Ramanujan's 691 congruence). To be more precise,  
given $k\ge 2$, a squarefree positive integer $N$, and an Atkin-Lehner eigensystem $\varepsilon$ of $\Gamma_0(N)$, we define  
\begin{equation}\label{Eisenstein series }
     \mathcal{E}^{(\varepsilon)}_{k,N}(z)\coloneqq  \sum\limits_{\substack{{d|N}\\}} \varepsilon(d) d^{k/2}E_k(dz),
\end{equation}
 where  $\displaystyle{E_k(z)}$ is the Eisenstein series of weight $k$ and level $1$. If $k=2$, we also assume that $\varepsilon(p)=-1$, for some $p\mid N$. Then   $\mathcal{E}^{(\varepsilon)}_{k,N}\in  \mathcal{M}_k^{(\varepsilon)}(N)$ (see, Section \ref{eisenstein series} for more details) and we have the following result.

%Let $f\in \mathcal{S}_k^{(\varepsilon)}(N)$ be a newform. It is important to note that if $f(z)\equiv  \mathcal{E}^{(\varepsilon)}_{k,N}(z) \, \modlam$ for some prime $\Lambda$ above $\ell$, then  $
  %      \overline{\rho}_{f,\Lambda} \simeq 1\oplus \overline{\chi}_\ell^{k-1}$. Moreover, the converse of this assertion is also true if $\ell> k+1$, which is underlined in the following lemma and a proof is given in Section \ref{proof of converse}.
    \begin{lemma}\label{Converse}
    With the same notation as in Conjecture \ref{conjecture}, let $f\in \mathcal{S}_k^{(\varepsilon)}(N)$ be a newform and  $\Lambda $ be a prime lying over $\ell$ in the coefficient field of $f$. Then
    $$
     \overline{\rho}_{f,\Lambda} \simeq 1\oplus \overline{\chi}_{\ell}^{k-1} \quad \textit{if and only if}\quad  f(z) \equiv \mathcal{E}_{k,N}^{(\varepsilon)}(z)\modlam.
     $$
        \end{lemma}
The direct implication is proved in Section \ref{proof of converse} while the reverse implication trivially follows from the Chebetarov density theorem.

\subsection{Main Results}
  The next two results determine the necessary and sufficient conditions for the existence of a newform of level having a product of two primes congruent to an Eisenstein series of the same level. We emphasize that our results are true for $k=2$ as well. In the following theorem and the ensuing results, we use the notation $\phi$ to denote the familiar Euler-$\phi$ function.

\begin{theorem}\label{result for N=p}
     Let $k\geq 2$ be even, $p$, $q$ and $\ell$ be distinct primes, and $\varepsilon$ be an Atkin-Lehner eigensystem for $\Gamma_0(pq)$. Assume that $\ell\ge \max\{5,k-1\}$, $\ell\neq k+1$ and 
         $  \ell\nmid \frac{B_k}{2k}\phi(pq) (q+1)$.
     Suppose that
    \begin{align}\label{divis}
      \ell\mid (1+\varepsilon(p) p^{k/2}) \quad \oq{and}\quad \ell\mid (1+\varepsilon(q) q^{k/2-1}).
      \end{align} 
         Then there exists a newform $f\in \mathcal{S}_k^{(\varepsilon)}(pq)$  such that 
\begin{equation}\label{cong_pq}
      f(z)\equiv \mathcal{E}_{k,pq}^{(\varepsilon)}(z) \modlam ,
\end{equation}
    for some prime $\Lambda $ over $\ell$ in a sufficiently large field. If $k=2$, then under the same assumptions as above but removing the condition $\ell\nmid(q+1)$, there exists a newform $f \in \mathcal{S}_2(pq)$ such that
$$\overline{\rho}_{f,\Lambda} \simeq 1\oplus \overline{\chi}_{\ell}.$$
     
\end{theorem}

The above theorem follows as a consequence of a more general result stated in Theorem \ref{newform_dp} and Remark \ref{k=2} when considering $N = q$.

\begin{comment}
\begin{theorem}\label{result for N=p}
    Let $k\geq 2$ be an even integer. Let $\ell,p, q$ be %distinct
    primes such that $\ell\ge \max\{5,k-1\}$,  
  $\ell\neq k+1$, and $\ell\nmid \frac{B_k}{2k}\phi(pq)(q+1)$. Let $\varepsilon$ be an Atkin-Lehner eigensystem for $\Gamma_0(pq)$. 
    %For $k=2$, we also assume that $\varepsilon(p)=-1$ and $\varepsilon(q)=-1$. 
Suppose
    \begin{align}\label{divis}
      \ell\mid (1+\varepsilon(p) p^{k/2}) \quad \oq{and}\quad \ell\mid (1+\varepsilon(q) q^{k/2-1}).
      \end{align} 
         Then there exists a newform $f\in \mathcal{S}_k^{(\varepsilon)}(pq)$  such that 
\begin{equation}\label{cong_pq}
      f(z)\equiv \mathcal{E}_{k,pq}^{(\varepsilon)}(z) \modlam ,
\end{equation}
    for some prime ideal $\Lambda $ lying over $\ell$ in a sufficiently large field. % generated by the coefficients of all the newforms of level $pq$. In particular, the $\oq{mod}\,\ell$ Galois representation $1\oplus \overline{\chi}_{\ell}^{k-1}$ is modular of level $pq$.
\end{theorem}
\end{comment}
\begin{remark}\label{remark_divis}
    Using \eqref{divis}, we note that for $k\ge 4$ some of the assumptions of  $\ell\nmid\phi(pq)(q+1)$ in the above theorem are vacuously true in certain cases listed below.
     \begin{enumerate}
         \item 
         The assumptions $\ell\nmid (q-1)$ and $\ell\nmid (p-1)$ hold automatically if $\varepsilon(q)=1$ and $\varepsilon(p)=1$, respectively. 
         \item 
         The assumption $\ell\nmid (q^2-1)$ holds automatically if $\varepsilon(q) = 1$ and $k\equiv 2\pmod{4}$.   
         \item 
         The assumption $\ell\nmid (q+1)$ holds automatically if $\varepsilon(q) = -1$ and $k\equiv 0\pmod{4}$. 
     \end{enumerate}
    In particular, if $\varepsilon(p) = 1$, $\varepsilon(q) = 1$ and $k\equiv 2\pmod{4}$, then  we only need to assume that $\ell\nmid \frac{B_k}{2k}$ instead of $\ell\nmid \frac{B_k}{2k}\phi(pq)(q+1)$.
    Furthermore, for $k=2$,  the assumptions in the above theorem force us to take $\varepsilon(p)=1$ and $\varepsilon(q)=-1$.
 \end{remark}
 By comparing the constant terms, as well as the $p$th and $q$th Fourier coefficients in the congruence \eqref{cong_pq}, we observe that the conditions  $
\ell \mid (1 + \varepsilon(p) p^{k/2})$ and $\ell \mid (1 + \varepsilon(q) q^{k/2 - 1})$  
become necessary for the congruence, under a mild assumption stated below.

\begin{theorem}\label{converse for level pq}
     Let $k\geq 2$ be even, $p, q$ and $\ell$ be distinct
    primes, and $\varepsilon$ be an Atkin-Lehner eigensystem for $\Gamma_0(pq)$. Assume that $\ell\nmid \frac{B_k}{2k}(1+\varepsilon(q)q^{k/2})$. For $k=2$, we also assume that $\varepsilon(q)=-1$.   
    %For $k=2$, we also assume that $\varepsilon(p)=-1$ and $\varepsilon(q)=-1$. 
    If  there is a newform $f\in \mathcal{S}_k^{(\varepsilon)}(pq)$ such that $$  f(z)\equiv \mathcal{E}_{k,pq}^{(\varepsilon)}(z)\modlam$$ for some prime ideal $\Lambda$ lying over $\ell$ in the coefficient field of $f$, then 
    $$
      \ell\mid (1+\varepsilon(p) p^{k/2}) \quad \oq{and}\quad \ell\mid (1+\varepsilon(q) q^{k/2-1}).
      $$
\end{theorem}
In view of Remark \ref{remark_divis}, combining Theorems \ref{result for N=p} and \ref{converse for level pq}, we have the following result.

\begin{corollary}\label{cor_conj}
Conjecture \ref{conjecture} is true for  $N=pq$ if $\ell\nmid \frac{B_k}{2k}(1+\varepsilon(q)q^{k/2})$ and either of the following conditions hold:
\begin{enumerate}
  \item $\varepsilon(p) = 1$, $\varepsilon(q) = 1$, and $\ell\nmid(q+1).$

    \item $k\equiv 2\pmod{4}$, $\varepsilon(p)=1$, and $\varepsilon(q) = 1$.

%    \item \textcolor{red}{I think, we should remove it} $k\equiv 2\pmod{4}$, $\varepsilon(q)=1$, and $\ell\nmid\phi(p).$

      \item $k\equiv 0\pmod{4}$, $\varepsilon(q) = -1$, and $\ell\nmid\phi(pq).$  
    
\end{enumerate}

\end{corollary}

The following theorem essentially gives sufficient conditions under which the $\oq{mod}\,\ell$ Galois representation $1\oplus \overline{\chi}_{\ell}^{k-1}$ is modular of level having at least two prime factors. Particularly, when $N$ is a prime, the subsequent result proves  Theorem \ref{result for N=p}. 

\begin{theorem}\label{newform_dp}
Let $k\geq 2$ be even, $p, \ell, p_1, p_2, \dots ,p_t$ be distinct primes and $N=p_1\dots p_t$. Assume that $\ell\ge \max\{5,k-1\},$  $\ell\neq k+1$ and  $\ell\nmid \frac{B_k}{2k}\phi(Np)(p_r+1)$ for some $1\le r\le t$.
For an Atkin-Lehner eigensystem $\varepsilon$ for $\Gamma_0(Np)$, suppose 
\begin{align}\label{divi}
 \ell \mid(1+\varepsilon(p) p^{k/2}) \quad {\rm and} \quad
\ell\mid(1+\varepsilon(p_i) p_i^{k/2-1}) ~{{  for ~each} ~1\leq i\leq t}. %\quad {\rm if~} k\ge 4.
\end{align}
Then
there exists a newform $f \in \mathcal{S}_k^{(\varepsilon)}(dp)$ for some $1 < d\mid N$  such that 
$$
f(z) \equiv \mathcal{E}_{k,dp}^{(\varepsilon)}(z)\modlam,
$$
for some prime ideal $\Lambda$ lying over $\ell$. 
%Moreover, if we also assume that $\ell\nmid (p_i+1)$,
%and $\ell\nmid (p_j+1)$, 
%for some fixed $i$, 
%then the Atkin-Lehner eigensystem of the above $f$ is $\varepsilon$ and 
%$$f \equiv \mathcal{E}_{k,dp}^{(\varepsilon)}\, \mathrm{mod}\, \Lambda. $$ 
\end{theorem}
%We remark that  in the above theorem for $k\ge 4$ the condition $\ell\nmid (p_r+1)$ for some $r$ becomes unnecessary if we take $\varepsilon(p_r)= (-1)^{k/2-1}$.
\begin{remark}\label{k=2}
We emphasize that, in the proof of the above theorem presented in  Section \ref{proof_main}, the assumption that  $\ell\nmid(p_r+1)$ is crucially used in Subcase (b) to ensure that the newform $f$ (with reducible $\oq{mod\,}\ell$ representation) is not of level $p_r$ and so it is of level $pp_r$. However, if $k=2$, then this is vacuously true since $\ell\nmid (p_r-1)$ (\cite[Proposition 5.12]{maz}). If $k=2$, then under the same assumptions as in Theorem \ref{newform_dp} but removing the condition $\ell\nmid(p_r+1)$, there exists a newform  $f \in \mathcal{S}_2(dp)$ such that
$$\overline{\rho}_{f,\Lambda} \simeq 1\oplus \overline{\chi}_{\ell}.$$
\end{remark}
 A key step in the proof of Theorem \ref{newform_dp} is to obtain   Theorem \ref{existence of eigenform}, which essentially states that the divisibility conditions \eqref{divi} are sufficient for the existence of an eigenform  in $\mathcal{S}_k(Np)$ that is congruent to   $\mathcal{E}_{k,Np}^{(\varepsilon)}(z)$ modulo a prime over $\ell$. Further, we use the strong multiplicity one theorem for modular forms and some classical results involving Galois representations attached to modular forms to get a newform with desired properties.

 \pagebreak 
\subsection{Applications} We provide the following two applications of Theorem \ref{result for N=p}.
\subsubsection{\bf Congruences for weight $2$ and $\ell+1$}
We obtain the following result as a consequence of Theorem \ref{result for N=p}.

\begin{comment}
 \begin{theorem}\label{density1}
 Let $k\ge 2$ be an integer and $\ell \geq \max \{ 5, k-1\}$ be a prime such that $\ell \neq k+1$ and   $k \equiv a \pmod {\ell-1}$ with $\ell\nmid \frac{B_a}{2a}$. Then there exist explicit sets of primes $\mathcal P$ and $\mathcal Q$ of densities $\frac{1}{\ell}$ and $\frac{\ell-2}{\ell}$ respectively such that for any $p\in\mathcal P$ and $q\in \mathcal Q$, there exists a newform $f \in \mathcal{S}_{k}^{(\varepsilon)}(pq)$ for which
$$f(z) \equiv \mathcal{E}_{k,pq}^{(\varepsilon)}(z)\modlam $$
where $\Lambda$ is a prime ideal in a sufficiently large number field lying over $\ell$ and  the Atkin-Lehner eigensystem $\varepsilon$ for level $pq$ is given by $ \varepsilon(p)=-p^{k/2}\modl$  and $ \varepsilon(q) = -q^{k/2-1}\modl.
$
\end{theorem}
\end{comment}

\begin{theorem}\label{density}
Let $\ell\ge 5$ be a prime and $k=2$ or $\ell+1$. There are explicit sets of primes $\mathcal P$ and $\mathcal Q$ of densities $\frac{1}{\ell}$ and $\frac{\ell-2}{\ell}$ respectively such that for any $p\in\mathcal P$ and $q\in \mathcal Q$, there exist a newform $f \in \mathcal{S}_{k}^{(\varepsilon)}(pq)$ and a prime ideal $\Lambda$ lying over $\ell$ in a sufficiently large number field for which
$$f(z) \equiv \mathcal{E}_{k,pq}^{(\varepsilon)}(z)\modlam,$$
where the Atkin-Lehner eigensystem  $\varepsilon$ is explicitly determined.
    
\end{theorem}

%{ Dieulefait et al. in \cite[Theorem 1] {dieule} showed the existence of weight $2$ newform of squarefree level with large coefficient field, as a corollary to above theorem we remark that for sufficiently large prime $\ell$ there exist a newform of weight $\ell+1$ and level equal to product of two distinct primes with arbitrarily large coefficient field.} 

%\begin{corollary}
%Let $A>0$ be a positive integer. For any large enough prime $\ell$, there exist distinct primes $p$ and $q$, a newform $f\in \mathcal{S}_k^{(\varepsilon)}(pq)$ where $k=2$ or $\ell+1$ and $\varepsilon$ is an Atkin-Lehner eigensystem such that $[K_f:\mathbb{Q}]>A$.
    
%\end{corollary}

\subsubsection{\bf A lower bound of the degree of coefficient fields}

For a normalized eigenform $f\in \mathcal{S}_k(N)$ with Fourier coefficients $a_f(n)$, let $K_f:\mathbb Q(a_f(n):n\ge 1)$ denote the number field attached to $f$. We define
\[
            d_k(N)^{new} \coloneqq {\max}\{[K_f:\mathbb{Q}]\,:\, f\in \mathcal{S}_k(N), \text{ $f$ is a newform} \} 
\]

It is an important and difficult problem to understand the growth of $d_k(N)^{new}$ as $k$ and $N$ are large.  %Serre \cite{ser2} proved that $d_k(N)^{new} \rightarrow \infty$ as $k+N\rightarrow \infty$. 
Tsaknias \cite{tsak} conjectured that for a fixed $k$,  $d_k(N)^{new}$ is of the order $N^{1-\epsilon}$, for any $\epsilon>0$. In this direction, there are many results in the literature (see the introduction of  \cite{bet}) but are far away from the conjectural bound of Tsaknias. 
The best known result is due to Bettin {et al.} \cite{bet} who proved that for $k\ge 2$ and $N\ge 1$, $d_k(N)^{new}\gg \frac{\log\log N}{2p_N}$, where $p_n$ is the smallest prime coprime to $N$. From Theorem \ref{result for N=p}, we obtain a lower bound for $d_2(pq)^{new}$ using similar arguments presented in the proof of \cite[Theorem 2]{bm}.  We give a brief outline of the proof.

Consider the set $\mathcal{N} = \{N = pq: p, q \text{ distinct primes }P^+((p+1,q+1))> N^{1/4}\}$, where $P^+(n)$ denotes the largest prime factor of $n$ with $P^+(1)=1$. Since,  \cite[Theorem 2]{LMPM} remains true if we replace $p_i-1$ by $p_i+c$ for any $c\in \bb{Z}$ in the definition of $\mathcal{A}_{k,c}$ in loc. cit., we obtain 
$$
 |\{N\le x:N\in \mathcal N\}|\gg \frac{x^{1/2}}{(\log x)^3}.
$$
For $N\in \mathcal{N}$, take $\ell = P^{+}\left(\mathrm{gcd}(p+1,q+1)\right)$, then $\ell$ satisfies the conditions in \eqref{divis}. Therefore, by Theorem \ref{result for N=p}, there exist a newform $f\in \mathcal{S}_2(pq)$ and a prime ideal $\Lambda$ over $\ell$ such that $\overline{\rho}_{f,\Lambda} \simeq 1\oplus \overline{\chi}_{\ell}$. Therefore $a_f(2)\equiv 3\modlam$ giving that $\ell\mid {\rm norm}_{K_f/\mathbb Q}(a_f(2)-3)$ and hence $\ell\le  {\rm norm}_{K_f/\mathbb Q}(a_f(2)-3)$. Further by Ramanujan's bound, we have  $\ell\le (1+\sqrt{2})^{2[K_f:\bb{Q}]}$. Using the facts that  $\ell> (pq)^{1/4}$ and $d_2(pq)^{new}\ge [K_f:\bb{Q}]$, we get the following lower bound {which improves the result of Dieulefait et al. \cite[Theorem 1]{dur}}.

%Further since $|a_f(2)|\le 2\sqrt{2}$, we have  $|a_f(2)-3|\le (1+\sqrt{2})^2$. If we set $[K_f:\bb{Q}]=d$ then $\ell\le norm(a_f(2)-3)\le (1+\sqrt{2})^{2d}$. Recalling that $\ell> (pq)^{1/4}$ and $d_k(pq)^{new}\ge d$, we get the following lower bound {which improves the result of Dieulefait et al. \cite[Theorem 1]{dur}}.
\begin{corollary}
    For any $N=pq\in \mathcal{N}$, we have $$d_2(N)^{new} \ge \frac{1}{8}\log N.$$
\end{corollary}

%If $\mathcal{N}' = \{N = p_1p_2\dots p_t: p_i's \text{ are  distinct primes and } P^+((p_1+1,p_2+1,\dots ,p_t+1))> N^{1/4}\}$, then using similar arguments as above and Theorem \ref{existence of eigenform}, we can easily obtain the following.
%\begin{equation}
 %           d_k(N) \coloneqq {\max}\{[K_f:\mathbb{Q}]\,:\, f\in \mathcal{S}_k(N), \text{ $f$ is an eigenform} \} \ge   \frac{1}{8k} \log N, \,\, \text{ for any } N\in \mathcal N'.
%\end{equation}

\subsection{Admissible tuples}
Ribet defined the notion of an admissible tuple for weight $2$ newforms by using the $U_p$ operators. If $f\in \mathcal{S}_2(N)$ is a newform, then $U_p(f) = -W_p(f)$ for any $p\mid N$. We use this observation to define the admissibility of $t$-tuples  for newforms of arbitrary weights by using the Atkin-Lehner operators. 
\begin{definition}
A $t$-tuple $(p_1, \ldots, p_t)$  of distinct primes is {\rm admissible for $s$ and weight $k$}, where $0\le s\le t$, if there exists a newform $f\in \mathcal S_k(p_1\ldots p_t)$ with reducible $\oq{mod}\,\ell$ Galois representation such that $W_{p_i}f = -f$ for all $1\le i \le s$ and $W_{p_i}f = f$ for all $s+1\le i \le t$.
\end{definition}
 
 Ribet proved that if a $t$-tuple is admissible for $s<t$ and weight 2, then $\ell \mid (p_i+1)$ for $s+1\le i\le t$ (\cite[Theorem 1.2]{yoo}). Yoo gave sufficient conditions for a $t$-tuple to be admissible for $s$ and weight 2, where $s$ satisfies certain assumptions \cite[Theorem 1.3]{yoo}.

Extending the notion of the admissibility for any weight $k\ge 2$, our results provide a necessary and sufficient criterion for the admissibility of a $2$-tuple which is stated below.  Let $\ell\ge  5$, $p$ and $q$ be distinct primes, and $k\ge 2$  be an even integer such that $\ell\nmid \frac{B_k}{2k}$. 

\underline{\bf Case (i): $k=2$}. A $2$-tuple $(q,p)$ is admissible for $s=1$ and weight $2$ iff $\ell\mid (1+p)$ provided $\ell\nmid \phi(pq)(q+1)$ (see Remark \ref{remark_divis}). 

\underline{\bf Case (ii): $k\ge 4$}. 
In the following two tables, we present the necessary and sufficient conditions for the admissibility of $2$-tuple $(p,q)$ for $s=0,1,2$ and weight $k\ge 4$.

  \begin{comment}
\begin{center}
    \begin{tabular}{|p{2.3cm}|p{2.1cm}|c|p{5.5cm}|}
    \hline
    Weight & Admissibility of $(p,q)$ for & Assumption on $\ell$ & Necessary and sufficient condition \\
    \hline
    \multirow{4}{60em}{$k\equiv 0 \pmod{4} $} & $s=0$ & $\ell\nmid (1+q^{k/2})(1+q)$  & $\ell\mid (1+p^{k/2})$ and $\ell\mid (1+q^{k/2-1})$ \\ [0.2cm] \cline{2-4} &  $s=1$ & $\ell\nmid \phi(pq)(q+1)(1+q^{k/2})$ & $\ell\mid (1-p^{k/2})$ and $\ell\mid (1+q^{k/2-1})$\\ [0.2cm] \cline{2-4} &  $s=2$ & $\ell\nmid \phi(pq)(1-q^{k/2})$ & $\ell\mid (1-p^{k/2})$ and $\ell\mid (1-q^{k/2-1})$\\ [0.2cm]
    \hline

    \multirow{4}{60em}{$k\equiv 2 \pmod{4} $} & $s=0$ & $\ell\nmid (1+q^{k/2})$  & $\ell\mid (1+p^{k/2})$ and $\ell\mid (1+q^{k/2-1})$ \\[0.2cm] \cline{2-4} &  $s=1$ & $\ell\nmid \phi(p)(1+q^{k/2})$ & $\ell\mid (1-p^{k/2})$ and $\ell\mid (1+q^{k/2-1})$\\[0.2cm] \cline{2-4} &  $s=2$ & $\ell\nmid \phi(pq)(1+q)(1-q^{k/2})$ & $\ell\mid (1-p^{k/2})$ and $\ell\mid (1-q^{k/2-1})$\\ [0.2cm]
    \hline

    \end{tabular}
\end{center}
\end{comment}

\begin{table}[h]
    \begin{tabular}{|m{2cm}|m{4cm}|m{4.8cm}|}
    \hline
    $(p,q)$ is admissible for  & {Assumptions on $\ell$ for admissibility for $k$} & Necessary conditions for admissibility\\%\cline{2-3}
     [0.1cm]\hline

     $s=0$  &$\ell\nmid (1+q^{k/2})$   & $\ell\mid (1+p^{k/2})$; $\ell\mid (1+q^{k/2-1})$\\ [0.22cm]
     \hline

     $s=1$ &$\ell\nmid (1+q^{k/2})$ &$\ell\mid (1-p^{k/2})$;  $\ell\mid (1+q^{k/2-1})$\\[0.22cm]
     \hline

     $s=2$ &$\ell\nmid (1-q^{k/2})$   &$\ell\mid (1-p^{k/2})$; $\ell\mid (1-q^{k/2-1})$\\[0.22cm]
     \hline
      \end{tabular} \vspace*{0.15cm}
      \caption{Necessary conditions.\label{table 1}}

    \end{table}

\begin{table}[h]
    \begin{tabular}{|p{2cm}|p{3.2cm}|p{3.2cm}|m{4.8cm}|}
     \hline
    $(p,q)$ is admissible for  & \multicolumn{2}{|c|}{Assumptions on $\ell$ if} & Sufficient conditions for admissibility\\ [-3mm] %\cline{2-3}
     & $k\equiv 0\pmod{4}$ & $k\equiv 2\pmod{4}$ & \\ [0.1cm]\hline

     $s=0$ &$\ell\nmid (1+q)$ & No assumption   & $\ell\mid (1+p^{k/2})$; $\ell\mid (1+q^{k/2-1})$\\ [0.22cm]
     \hline

     $s=1$ &$\ell\nmid \phi(p)(q+1)$ &$\ell\nmid \phi(p)$ &$\ell\mid (1-p^{k/2})$;  $\ell\mid (1+q^{k/2-1})$\\[0.22cm]
     \hline

     $s=2$ &$\ell\nmid \phi(pq)$  & $\ell\nmid \phi(pq)(1+q)$ &$\ell\mid (1-p^{k/2})$; $\ell\mid (1-q^{k/2-1})$\\[0.22cm]
     \hline

    \end{tabular}\vspace*{0.15cm}
    \caption{Sufficient conditions}\label{table 2}

    \end{table}

The first column in each table contains the three possible values of $s$. The second column lists the assumptions on $\ell$ which need to be satisfied depending on the weights and values of $s$. The necessary conditions for the corresponding values of $s$ are listed in the third column of Table 1, whereas the sufficient conditions are presented in the third column of Table 2, which follows directly from Theorems \ref{result for N=p} and \ref{converse for level pq} respectively. For example, if $s=0$, we have $\varepsilon(p)=\varepsilon(q)=1$ in Theorem \ref{result for N=p}, and so the assertion for $s=0$ in Table 2 follows.

\begin{comment}

\begin{enumerate}
    \item 
    The condition $\ell\mid (1+p^{k/2})$ and $\ell\mid (1+q^{k/2-1})$ is necessary and sufficient for admissibility of $(p,q)$ for $s=0$ and weight $k$ provided $\ell\nmid \frac{B_k}{2k}(1+q^{k/2})(1+q)$ when $k\equiv 0\pmod{4}$ and $\ell\nmid \frac{B_k}{2k}(1+q^{k/2})$ when $k\equiv 2\pmod{4}.$

    \item The condition $\ell\mid (1-p^{k/2})$ and $\ell\mid (1+q^{k/2-1})$ is necessary and sufficient for admissibility of $(p,q)$ for $s=1$ and weight $k$ provided $\ell\nmid \frac{B_k}{2k}\phi(pq)(q+1)(1+q^{k/2})$ when $k\equiv 0\pmod{4}$  and  $\ell\nmid \frac{B_k}{2k}\phi(p)(1+q^{k/2})$ when $k\equiv 2\pmod{4}$.

    \item The condition $\ell\mid (1-p^{k/2})$ and $\ell\mid (1-q^{k/2-1})$ is necessary and sufficient for admissibility of $(p,q)$ for $s=2$ and weight $k$ provided $\ell\nmid \frac{B_k}{2k}\phi(pq)(1-q^{k/2})$ when $k\equiv 0\pmod{4}$  and $\ell\nmid \frac{B_k}{2k}\phi(pq)(1+q)(1-q^{k/2})$ when $k\equiv 2\pmod{4}.$
\end{enumerate}

\end{comment}

\subsection{Layout}
The second section of this paper deals with the basic facts about newforms, and $\oq{mod}\,\ell$ Galois representations. In the next section, we compute the Fourier coefficients of the Eisenstein series and its behaviour under the action of Hecke operators. The next four sections contain proofs of our results. In the last section, we give a few numerical examples to demonstrate our results.

 \pagebreak 

\section{Preliminaries}

In this section, we gather some notations, definitions, and some well-known results that will be used later in the article.

\subsection{Notation}
We keep the notations established in the previous section and introduce additional notations here. Recall that $k\ge 2$ denotes an even integer, and $N$ denotes a squarefree positive integer. Throughout the article, $p,q,p_1,p_2,\ldots,p_t$ and $\ell$ stand for distinct primes.
%except in Section \ref{examples}, where we use $q = e^{2 \pi i z}$ in the Fourier series expansions. 
For a fraction $a/b$, we write $\ell\mid ({a}/{b})$ to mean that $\ell$ divides the numerator of the reduced fraction.
We denote the complex vector space of modular forms with weight 
$k$ for the congruence group $\Gamma_0(N)$ with the trivial nebentypus as $\mathcal{M}_k(N)$, and its subspace of cusp forms is denoted by $\mathcal{S}_k(N)$. For an newform $f$, we use $K_f$ to denote the coefficient field of $f$.  

We let $\overline {\mathbb Q}$ to denote the algebraic closure of $\mathbb Q$ in $\mathbb C$, and $\overline {\mathbb Z}$ to denote the integral closure of $\mathbb Z$ in $\overline{\mathbb Q}$. For any prime $\ell$, we fix an algebraic closure $\overline{\mathbb Q}_\ell$ of $\mathbb Q_\ell$, and let $\overline {\mathbb Z}_\ell$ denote the integral closure of $\overline {\mathbb Z}_\ell$ in $\mathbb Q_\ell$. Additionally, we fix an embedding  $\overline {\mathbb Q} \hookrightarrow \overline {\mathbb Q}_\ell$ which naturally gives rise to an embedding $\overline{ \mathbb Z} \hookrightarrow \overline {\mathbb Z}_\ell$. The notation  $ \overline {\mathbb F}_\ell$ is used to denote the residue field of $\overline {\mathbb Z}_\ell$, serving as an algebraic closure for the field $\mathbb F_\ell$ consisting of $\ell$ elements.

\subsection{Newform theory}\label{newform theory}
%Let $N$ be a squarefree integer and $\mathcal{S}_k(N)$ be the complex vector space of cusp forms of trivial character for the congruence group  $\Gamma_0(N)$. 
For a given weight $k$ and level $N$, let $\mathbb T_k(N)$ be the \emph{Hecke algebra} which is the $\mathbb Z$-subalgebra of ${\rm End}(\mathcal M_k(N))$  generated by Hecke operators $T_p$ for primes $p$. For the remainder of this section, the letter $p$ will usually denote a prime number such that $p\mid N$, and the letter $q$ will usually denote a prime number
such that $q \nmid N$. Hereafter, to dinstinguish the operators $T_p$  and $T_q$, we set $T_p = U_p$ (for $p\mid N$) so that the algebra $\bb{T}_k(N)$ is generated by $T_q$  and $U_p$.
These operators are multiplicative, stabilize the spaces $\mathcal M_k(N)$ and $\mathcal S_k(N)$, and satisfy the following recurrence relation for any positive integer $r\ge 2$,
\[
        T_{q^r} =  T_q T_{q^{r-1}} - q^{k-1} T_{q^{r-2}}\quad {\rm and} \quad U_{p^r} =  U_p^r.
\]
%Also, for prime $p\mid N$, define $U_p \coloneqq T_p$ and for $q\nmid N$, $U_q \coloneqq T_q - q^{k-1}B(q)$, where 

For a positive integer $d$,  the duplication operator $B_d: \mathcal M_k(M) \rightarrow \mathcal M_k(Md)$ is defined by 
\begin{equation}\label{Bd}
B_d:  f(z) \mapsto f(dz),
 \end{equation}
 which also maps a cusp form to a cusp form.
Now, for primes $q\nmid N$, we set 
$$U_q \coloneqq T_q - q^{k-1}B_q.$$
For more details on the Hecke algebra, we refer the reader to \cite{diashur}.

\begin{definition}
A modular form $f(z) = \sum_{n\ge 1} a_f(n){e}^{2\pi i n z}\in \mathcal{S}_k(N)$ is called an \emph{eigenform} if it is an eigenfunction for all the operators in $\mathbb T_k(N)$.% or equivalently, if there is a ring homomorphism $\lambda  : \mathbb T_k(N)\rightarrow C$ such that $Tf = \lambda(T)f$ for all $T \in \mathbb T_k(N)$. 
\end{definition}
If $f$ is an eigenform then $a_f(1)\neq 0$ and we may assume that $a_f(1) = 1$. In such a case we call $f$ a \emph{normalized eigenform}.

Now, we review some basics of newform theory developed by Atkin and Lehner \cite{al}.
The map $B_d$, defined in \eqref{Bd}, embeds the space $\mathcal{S}_k(N/d)$ inside $\mathcal{S}_k(N)$. Thus the space of \emph{oldforms}  %${\mathcal{S}_{k}(N)}^{old}$
is defined as 
\begin{equation}\label{oldforms}
{\mathcal{S}_{k}(N)}^{old}\coloneqq \oplus_{M|N, M<N}\oplus_{d|N/M}B_d\mathcal{S}_{k}(M).
\end{equation}

%The space $\mathcal{S}_k(N)$ is equipped with the Petertsson inner product. 
The space of \emph{newforms} ${\mathcal{S}_{k}(N)}^{new}$  is defined as the orthogonal complement of ${\mathcal{S}_{k}(N)}^{old}$ in the space $\mathcal{S}_k(N)$ with respect to the Petersson inner product.
%and we can write 
%\begin{equation}\label{direct sum}
%     \mathcal{S}_k(N) = {\mathcal{S}_{k}(N)}^{old} \oplus^{\perp} {\mathcal{S}_{k}(N)}^{new}
%\end{equation}
A form $f\in {\mathcal{S}_{k}(N)}^{new} $ is called a \emph{newform} if it is an eigenform. %, i.e.,
%\[
 %       T_q(f) = a_f(q)f \quad \oq{and}\quad U_p(f) = a_f(p)f
%\]
 %for all primes $p,q$ such that $p\mid N$ and $q\nmid N$. 
 We always assume a newform to be normalized.

 Let $f\in \mathcal{S}_{k}(N)^{old}$. Using \eqref{oldforms} and appealing to the multiplicity one theorem, we deduce that there exists a newform $h\in \mathcal{S}_k(M)^{new}$  for $M\mid N$  such that 
 \[
    f(z) = h(z) + \sum\limits_{1<d\mid N/M} \alpha(d)h(dz), \quad {\rm where}~ \alpha(d)\in \bb{C}.
 \]
   
  %We use this expression in the proof of Theorem \ref{}

  %The newforms act as a building block of the space of the cusp forms as we see below.

The Atkin-Lehner operator $W_p$, acting on $\mathcal{M}_k(N)$ for each $p\mid N$ and preserving $\mathcal{S}_k(N)$, is useful to characterize newforms and is defined by 
 \[
        W_pf(z) = p^{-k/2}z^{-k} f\left(\frac{-1}{pz}\right).
 \]
It is an involution on $\mathcal{S}_k(N)^{new}$ having eigenvalues $\{\pm 1\}$. It is known that an eigenform $f\in \mathcal{S}_k(N)$ is a newform iff $f$ is an eigenfunction of $W_p$ for $p\mid N$.
 Let $\varepsilon$ be an Atkin-Lehner eigensystem for $\Gamma_0(N)$ as defined in the introduction. %meaning a multiplicative function  $\varepsilon:\mathcal{P}_N\rightarrow \{\pm 1\}$, where $\mathcal{P}_N$ denotes the set of positive divisors of $N$ with $\varepsilon(1)=1$.
% Recall that $\varepsilon$ can also be considered as an Atkin-Lehner eigensystem for $\Gamma_0(M)$ for any $M\mid N$.
 We define the space of newforms of level $N$ with the Atkin-Lehner eigensystem $\varepsilon$ as
 $${\mathcal{S}_{k}^{(\varepsilon)}(N)}^{new} = \langle f\in  {\mathcal{S}_{k}(N)}^{new}:  W_pf(z) = \varepsilon(p) f(z) \rangle. $$ 
 Moreover, if $f\in {\mathcal{S}_{k}^{(\varepsilon)}(N)}^{new} $ is a newform, the by Ogg's result \cite[Theorem 2]{ogg}, we have
%\textcolor{red}{\bf Give the exact Theorem number here and throughout the paper}
\begin{equation}\label{ogg's result}
    a_f(p) = -\varepsilon(p) p^{k/2-1} \quad {\rm for~any}~p\mid N.
\end{equation}
%We state the Diamond's generalization of weight $k=2$ level raising result by Ribet.
We end this section by stating the following level raising theorem due to Diamond and Ribet.

\begin{theorem} \cite[Theorem 1]{dia}\label{level raising theorem}
Let $k\ge 2$ be an integer and  $g\in \mathcal{S}_k(N)$ be a newform of weight $k$ and level $N$. Let $p$ and $\ell$ be primes  such that $p\nmid N\ell$ and $\ell\nmid \frac{1}{2} \phi(N)Np(k-2)!$ and $\Lambda$ be a prime ideal lying over $\ell$ in sufficiently large  field. Then there exists a newform $f\in S_k(dp)$ for $d\mid N$ such that 
\[
        a_g(q) \equiv a_f(q) \modlam  \quad \oq{for~all~} q\nmid Np
\]
if and only if
\[
           a_f(p)^2 \equiv  p^{k-2}(1+p)^2 \modlam. 
\]
\end{theorem}

\subsection{$\boldsymbol{\oq{Mod}}$ $\ell$ modular forms}

Let $K$ be the compositum of coefficient fields of basis elements of $\mathcal{M}_k(N)$ and $\mathcal{O}_K$ the associated number ring.
For a fixed prime $\ell$ with $\ell\nmid N$, let $\Lambda$ be a prime ideal lying over $\ell$ in $\mathcal{O}_K$. Denote the localization of the ring  $\mathcal{O}_{K}$ at a prime ideal $\Lambda$ by 
    $\mathcal{O}_{K_{(\Lambda)}}$. An element of the ring $\mathcal{O}_{K_{(\Lambda)}}$ is called $\Lambda$-integral. We  denote the space of \emph{$\oq{mod\,\ell}$ modular forms} of weight $k$ and level $N$ by $\mathcal{M}_k(N,\overline{\mathbb{F}}_{\ell})$ and define it as 
\[ \mathcal{M}_k(N,\overline{\mathbb{F}}_{\ell}) \coloneqq 
    \left\{\bar{f}(z) = \sum\bar{a}_f(n)q^n : f\in \mathcal{M}_k(N) \text{ and } a_f(n)\in\mathcal{O}_{K_{(\Lambda)} } \text{for some prime $\Lambda|\ell$} \right\},
\]
where $\bar{a}_f(n)$ denotes the reduction $a_f(n)$ modulo $\Lambda$ and similarly we can define the space of \emph{$\oq{mod\,\ell}$ cusp forms} of weight $k$ and level $N$, denoted by $\mathcal{S}_k(N,\overline{\mathbb{F}}_{\ell})$. {For more details on mod$\,\ell$ modular forms, we refer the reader to Section $7$ of \cite[Chapter 10]{lang1}}.

Let $f(z) = \sum_{n\geq 1} a_f(n)e^{2\pi i n z} $ and $g(z) = \sum_{n\geq 1} a_g(n)e^{2\pi i n z}\in \mathcal{M}_k(N)$ be two forms, we call $f$ and $g$ to be \emph{congruent modulo $\Lambda$} and write $f(z)\equiv g(z)\modlam$ if $a_f(n)\equiv a_g(n)\modlam$ for all $n\geq 0.$
%The algebra of  $\oq{mod\,\ell}$ modular forms of level $N$ is given by the sum
%\[
 %       \mathcal{M}(N,\overline{\bb{F}}_{\ell}) =  \sum\limits_{k} \mathcal{M}_k(N,\overline{\mathbb{F}}_{\ell}) \subset \overline{\bb{F}}_{\ell}\llbracket q \rrbracket.
%\]
 Modular forms of different weights may be congruent modulo a prime ideal. %Two modular forms of weights $k$ and $k'$ in characteristic $0$ are congruent $\oq{mod\,\ell}$ iff $k\equiv k' \, \oq{mod\,\ell}$ or
 In fact, the intersection $\mathcal{M}_k(N,\overline{\mathbb{F}}_{\ell})\cap \mathcal{M}_{k'}(N,\overline{\mathbb{F}}_{\ell})$ is nonempty iff $k\equiv k'\pmod{\ell-1}$ (\cite{hpf} for $N=1$, \cite{katz} for $N>1$).

 %So for $i\in 2\bb{Z}/(\ell-1)\bb{Z}$, we take 
%\[
 %       \mathcal{M}(N,\bb{F}_{\ell})^{i}\coloneqq \sum\limits_{k\equiv i} \mathcal{M}_k(N,\bb{F}_{\ell}).
%\]
%Swinnerton-Dyer  \cite{SD} proved that
%\[
%        \mathcal{M}(N,\bb{F}_{\ell}) = \bigoplus\limits_{i\in 2\bb{Z}/(\ell-1)\bb{Z} }  \mathcal{M}(N,\bb{F}_{\ell})^{i}.
%\]

\begin{proposition}\label{lang}
    Let $f(z)= \sum_{n\geq 1} a_f(n)e^{2\pi i n z}$  and $g(z)= \sum_{n\geq 1} a_g(n)e^{2\pi i n z} \in \mathcal{M}_k(N)$ be two forms such that $a_f(n)$ and $a_g(n)$ are $\Lambda$-integral for all $n\ge 0$. Assume that 
    $$
        k \not\equiv 0 \pmod{\ell-1} \quad \mathrm{and}\quad a_f(n)\equiv a_g(n) \modlam  ~~~ \mathrm{for}~~ n\ge 1. 
    $$
    Then $f(z)\equiv g(z) \modlam$.
\end{proposition}
\begin{proof}
   We need to show that $a_f(0)\equiv a_g(0)\modlam$.  On the contrary, let us assume that $a_f(0)\not\equiv a_g(0)\modlam.$
   Define $h(z)\coloneqq a_f(0)-a_g(0)$. Then $h(z)$ is a modular form of weight $0$ and level $N$ and
   \[
            f(z)-g(z)\equiv h(z)\modlam.
   \]
   By the above result of Katz, it follows that 
   $
            k\equiv 0 \pmod{\ell-1},
   $
   which is a contradiction.
\end{proof}

Define $\mathcal{M}_k(N,\overline{\mathbb Z}_{\ell}):= \mathcal{M}_k(N,\bb Z)\otimes_\mathbb Z  \overline{\mathbb Z}_{\ell}$, where $\mathcal{M}_k(N,\mathbb Z)$ consists of forms in $\mathcal M_k(N)$ with integer Fourier coefficients. Similarly, we can define the space $\mathcal{S}_k(N,\overline{\mathbb Z}_{\ell})$. The Hecke operators defined earlier act on the space $\mathcal{S}_k(N,\overline{\mathbb Z}_{\ell})$ and also act
%hence
on the space $\mathcal{S}_k(N,\overline{\mathbb F}_{\ell})$ but with a small modification that the action of $T_\ell$ on $\mathcal{S}_k(N,\overline{\mathbb F}_{\ell})$ coincides with the action of $U_\ell$ which gives the notions of eigenforms in these spaces. Furthermore any form $f\in \mathcal{S}_k(N,\overline{\mathbb F}_{\ell})$ is the reduction of a form in $\mathcal{S}_k(N,\overline{\mathbb Z}_{\ell})$ as stated in the following lemma. 

\begin{proposition}[Carayol's lemma]\label{carayol}
Let $N\geq 1$, $\ell\ge 5$ be a prime not dividing $N$, and $k\geq 2$. %Let $\chi:(\mathbb{Z}/N\mathbb Z)^*\rightarrow \overline{\mathbb Z}_{\ell}^*$ be a character with $\chi(-1)=(-1)^k$,  and let $\overline{\chi}:(\mathbb{Z}/N\mathbb Z)^*\rightarrow \overline{\mathbb F}_{\ell}^*$ be its reduction $\rm mod~\ell$. 
Then, the reduction map $\phi: \mathcal{S}_k(N,\overline{\mathbb Z}_{\ell})\rightarrow \mathcal{S}_k(N,\overline{\mathbb F}_{\ell})$ is surjective.
\end{proposition}

By abuse of notation, we denote the commutative Hecke algebra acting on $\mathcal{S}_k(N, \overline{\bb{Z}}_\ell)$ by $\bb{T}_k(N)$. Then for an eigenform  $f\in \mathcal{S}_k(N, \overline{\bb{Z}}_\ell)$,  the reduction map $\psi_f: \bb{T}_k(N)\rightarrow \overline{\bb{F}}_{\ell}$ defined by 
  $\psi^{}_f(T_n) = a_f(n)$, for $n\ge 1$ is a homomorphism. We end this section by stating the following version of the Deligne-Serre lifting lemma.

\begin{proposition}[{Deligne-Serre lifting lemma}] \label{ds}
Let $k\geq 2$, $N\ge 1$ be positive integers, and $\ell$ be a prime. Let $f\in \mathcal{S}_k(N, \overline{\bb{Z}}_\ell)$ be an eigenform. Then there exist an eigenform $g(z) 
%= \sum\limits_{n\ge 1}a_g(n)q^{n}
\in \mathcal{S}_k(N)$ with Fourier coefficients $a_g(n)$ and a prime ideal $\Lambda$ lying over $\ell$ in the number ring $K_g$ such that
            $a_g(n)\equiv \psi_f(T_n)\modlam, \text{ for all } n\ge 1. $ 
\end{proposition}

\subsection{Galois representations}
The seminal works of Eichler, Shimura, and Deligne show that to a newform $f(z) = \sum\limits_{n\geq 1} a_f(n)q^{n}\in \mathcal{S}_k(N)$,  we can attach an $\ell$-adic Galois representation 
\[
		\rho_{f,\Lambda} : \mathrm{G}_\mathbb{Q} \rightarrow \mathrm{GL}_2(K_{f,\Lambda})
\]
 which is odd, unramified outside $N\ell$, and  for any $q\nmid N\ell$
\[
    \mathrm{tr}(\rho_{f,\Lambda}(\mathrm{Frob}_q)) = a_f(q) ~~ \text{  and } ~~ \mathrm{det}(\rho_{f,\Lambda}(\mathrm{Frob}_q))=q^{k-1},
\]
where $\Lambda$ is a prime ideal lying over $\ell$ in coefficient field $K_f$, $K_{f,\Lambda}$ is  the completion of $K_f$ at $\Lambda$ and $\oq{Frob}_q$ is the Frobenius element at $q$.  
%An $\ell$-adic representation $\rho_{f,\Lambda}$ is conjugate to a representation $\rho'_{f,\Lambda}: \mathrm{G}_\mathbb{Q} \rightarrow \mathrm{GL}_2(\mathcal{O}_{K_{f,\Lambda}})$.\\
In fact taking a suitable conjugate of $\rho_{f,\Lambda}$, we may assume that it is valued in $\mathcal{O}_{K_{f,\Lambda}}$ and further reduction modulo $\Lambda$ yields $\oq{mod\,\ell}$ Galois representation 
$$\overline{\rho}_{f,\Lambda}: \mathrm{G}_\mathbb{Q} \rightarrow \mathrm{GL}_2(\mathcal{O}_{K_{f,\Lambda}}/\Lambda)\xhookrightarrow \rm \oq{GL}_2(\overline{\bb{F}}_\ell).$$

Let $D_p$ denote the decomposition group of $G_\bb{Q}$ at a prime ideal lying over $p$.
For any algebraic integer $x$,  let 
\[
            \omega_x : D_p \rightarrow \overline{\bb{F}}_\ell^{\times}
\]
be the unique unramified character with values in $\overline{\bb{F}}_\ell^{\times}$ and $\omega_x(\oq{Frob}_p) = x\modl$. With these notations, we have the following result due to Langlands \cite{lan}.

\begin{theorem}[Langlands]\label{langland}
The restriction of $\overline{\rho}_{f,\Lambda}$ to $D_p$ is given by
\[
       \overline{\rho}_{f,\Lambda}|^{}_{D_p} \simeq \begin{pmatrix}
           \chi_{\ell}^{k/2} & *\\   &   \chi_{\ell}^{k/2-1}
       \end{pmatrix} \otimes \omega_{{a_f(p)}/{p^{k/2-1}}}.
\]

\end{theorem}

\section{An Eisenstein Series for Squarefree Level $N$}
Let $k\geq 2$ be an even integer, $N=p_1p_2\cdots p_t>1$ be a squarefree integer, and $\varepsilon$ be an Atkin-Lehner eigensystem for $\Gamma_0(N)$. Recall that 
%$\varepsilon$ is a multiplicative function  $\varepsilon:\mathcal{P}_N\rightarrow \{\pm 1\}$, where $\mathcal{P}_N$ denotes the set of positive divisors of $N$.  It is important to note that 
$\varepsilon$ can also be considered as an Atkin-Lehner eigensystem for $\Gamma_0(M)$ for any $M\mid N$ by restricting it to $\mathcal{P}_M$. So, by an abuse of notation, we also write $\mathcal M_k^{(\varepsilon)}(M)$ to denote the space of modular forms of weight $k$ for $\Gamma_0(M)$ with Atkin-Lehner eigensystem $\varepsilon$ and similarly for $\mathcal S_k^{(\varepsilon)}(M)$. 
%Furthermore, $\varepsilon$ can be extended to a completely multiplicative function on the set of positive integers such that
%$$ \varepsilon(d^2)=1, {\rm ~for~} d\in \mathbb \mathbb{N},$$
%which will be utilized in the proofs throughout this section without explicit specification. 
%For simplicity, throughout this section, we write
%$$
%\varepsilon_j=\varepsilon(p_j), ~{\rm for}~ 1\le j\le s. 
%$$
%Therefore, if $d =\prod\limits_{j=1}^{s}p_j^{i_j}$ is a divisor of $N$, where $i_j\in \{0,1\}$, we have
%$
%\varepsilon(d)= \prod\limits_{j=1}^{s}\varepsilon_j^{i_j} ~ {\rm and} ~\varepsilon(1)=1.$

\subsection{Eisenstein series}\label{eisenstein series}
We define an Eisenstein series of weight $k\ge 2$, level $N$ and Atkin-Lehner eigensystem $\varepsilon$ by
\begin{align}\label{eisenstein_N}
	\mathcal{E}_{k,N}^{(\varepsilon)}(z) &:= \prod_{i=1}^t(1+\varepsilon(p_i)W_{p_i})E_k(z) \notag\\
 &= E_k(z)+\sum\limits_{\substack{{1<d|N}}} \varepsilon(d) d^{k/2}E_k(dz),
\end{align}
where $E_k(z)$ is the %normalized 
Eisenstein series of weight $k$ and level $1$, defined as
$$
E_k(z)=-\frac{B_k}{2k}+\sum_{n\ge 1}\sigma_{k-1}(n)e^{2\pi i n z}.
$$
%whereas the Eisenstein series of weight $2$ and level $1$ \cite[Sec 7.2]{miy} is defined as 
%%  E_2(z) = -\frac{1}{8\pi^2} \lim\limits_{\epsilon \rightarrow 0}\sum\limits_{(c,d)\in \bb{Z}^2\backslash \{(0,0)\}}\frac{1}{(cz+d)^2 |cz+d|^{2\epsilon}}.
%\end{equation*}
We know that, $\mathcal{E}_{k,N}^{(\varepsilon)}(z)\in \mathcal M_k^{(\varepsilon)}(N)$, provided for $k=2$, $\varepsilon(p_j)=-1$ for some $j$. Since the Atkin-Lehner operator $W_{p_i}$ commutes with the $q$th Hecke operator $T_q$ for any $q\nmid N$, it follows that $\mathcal{E}_{k,N}^{(\varepsilon)}(z)$ is an eigenfunction for any Hecke operator $T_q$ with eigenvalue $1+q^{k-1}$. But in general, $\mathcal{E}_{k,N}^{(\varepsilon)}(z)$ is neither a cusp form nor an eigenfunction of $U_{p_i}$. Nevertheless, in the next two sections, we shall prove that for a suitable prime $\ell$, the reduction of $\mathcal{E}_{k,N}^{(\varepsilon)}(z)$ modulo $\ell$ is not only a $\oq{mod\,\ell}$ cusp form but also an eigenfunction for the Hecke operator $U_{p_i}$ for any $1\le i\le t$.

%and  $\varepsilon(d) = \prod\limits_{j=1}^{s}\varepsilon_j^{i_j}$ if $d =\prod\limits_{j=1}^{s}p_j^{i_j}$, $i_j\in \{0,1\}~~ {\rm for}~ 1\leq j\leq s $.

\subsection{Constant terms at the cusps of $\Gamma_0(N)$}
This section aims to compute the constant term of $\mathcal{E}_{k,N}^{(\varepsilon)}(z)$ at each of the cusps of $\Gamma_0(N)$ and identify a suitable prime $\ell$ such that the $\oq{mod\,\ell}$ modular form $\overline{\mathcal{E}}_{k,N}^{(\varepsilon)}(z)$ belongs to $\mathcal{S}_k(N,\overline {\mathbb F}_\ell)$ for each $k\ge 2$. 

Clearly, the Fourier  expansion of $\mathcal{E}_{k,N}^{(\varepsilon)}(z)$ at the cusp $\infty$ is given by 
$$\mathcal{E}_{k,N}^{(\varepsilon)}(z)
= -\frac{B_k}{2k} \prod\limits_{i=1}^{t}(1+\varepsilon(p_i) p_i^{k/2})+\sum\limits_{n\geq 1}\left( \sum\limits_{d|N}^{} \varepsilon(d) d^{k/2}  \sigma_{k-1}\left(\frac{n}{d}\right)\right) e^{2\pi inz},$$

where $\sigma_{k-1}\left(\frac{n}{d}\right)=0$ if $d\nmid n$. We denote its $n$th Fourier coefficient by $a(n)$, so 
$$
a(n)=\begin{cases}
    -\frac{B_k}{2k} \prod\limits_{i=1}^{t}(1+\varepsilon(p_i) p_i^{k/2}) & n=0\\
    \sum\limits_{d|N}^{} \varepsilon(d) d^{k/2}  \sigma_{k-1}\left(\frac{n}{d}\right)& n\ge 1.
\end{cases}
$$
A set of non-equivalent cusps of $\Gamma_0(N)$ is given by 
$$\left\{\frac{1}{M}:M\ge 1, M\mid N\right\}$$
and hence the number of cusps of $\Gamma_0(N)$ is $2^t$.

We first state the following result of Billery-Menares, which provides the constant term of $E_k(dz)$ at any cusp of $\Gamma_0(d)$ for $k\ge 4$.

\begin{proposition}\cite[Proposition 1.2]{bm}\label{bm_constant}
%\textcolor{red}{\bf check the Proposition number from the published paper and correct it throughout the paper}	
Suppose $k\geq 4$. Let  $\gamma = \begin{pmatrix}
	 u & v \\ x & w
	\end{pmatrix} \in \rm{SL}_2(\mathbb{Z})$ and $d\ge 1$ be an integer. The constant term of the Fourier expansion of $E_k(dz)|_{k}\gamma $ is  $-\dfrac{1}{d'^k}\dfrac{B_k}{2k}$ where $d' = \dfrac{d}{(x,d)}$ {and the slash-$k$ operator is defined as $f|_k \gamma \coloneqq (xz+w)^{-k} f\left(\frac{uz+v}{xz+w} \right)$}.
\end{proposition}

%We now prove that
The above result is also valid for $k=2$ as proved below.
\begin{proposition}
  For $k=2$,  Proposition \ref{bm_constant}  is  true.
\end{proposition}
\begin{proof}
It can be proved by using a similar argument used in the proof of \cite[Proposition 1.2]{bm}.
However, we provide a brief recapitulation of the argument here for the sake of conciseness.    By \cite[Lemma 10]{bm2}, the constant term of the Fourier expansion of $E_2(dz)|_2 \gamma$  is given by
    $$
       -\frac{1}{8\pi^2} \sum\limits_{\substack{(a,b)\in \bb{Z}^2 \backslash \{(0,0)\}\\adu+bx=0}}  \frac{1}{(adv+bw)^2}.
    $$
    If $u=0$, the condition $adu+bx=0$ reduces to $bx=0$ and hence  $b=0$. In this case, $d=d'$ and so the constant term becomes 
    $$ -\frac{1}{8\pi^2}\sum\limits_{a\in \bb{Z}\backslash \{0\}} \frac{1}{(adv)^2}=-\frac{1}{8\pi^2}\frac{2}{d'^{2}}  \zeta(2).$$
    Since $\zeta(2) = \pi^2 B_2$, it gives the desired result.\\
If $u\neq 0$, there are no solutions of the form $(a,0)\in \bb{Z}^2 \backslash \{(0,0)\} $ of $adu+bx=0$. Thus, for any $(a,b)\in \bb{Z}^2 \backslash \{(0,0)\}$ satisfying $adu+bx=0$, we have $adv+bw=\frac{b}{u}$. Hence, the constant term becomes
    $$  -\frac{1}{8\pi^2} \sum\limits_{\substack{a\in\bb{Z}, b\neq 0\\cdu+dx=0}}  \frac{1}{(adv+bw)^2}=  -\frac{1}{8\pi^2}\sum\limits_{\substack{b\neq 0\\d'u\mid b}} \left(\frac{u}{b}\right)^2 =  -\frac{1}{8\pi^2}\sum\limits_{b\neq 0}\frac{1}{(d'b)^2} =  -\frac{1}{8\pi^2}\frac{2}{d'^{2}}  \zeta(2)$$
which completes the proof.
\end{proof}

We are now ready to prove the following result which determines the constant term of $\mathcal{E}_{k,N}^{(\varepsilon)}(z)$ at each of the cusps of $\Gamma_0(N)$. 
\begin{theorem}
For each $k\ge 2$ and a positive divisor $M$ of $N$, the constant term of the Fourier series expansion of $\mathcal{E}_{k,N}^{(\varepsilon)}(z)$ at the cusp $1/M$ is  
$$ -\frac{B_k}{2k} {\varepsilon\left({\frac{N}{M}}\right) \left(\frac{M}{N}\right)^{k/2}} \prod\limits_{i=1}^{t}(1+\varepsilon(p_i) p_i^{k/2}).$$
 %where $d'=\frac{N}{M}$.	
\end{theorem}
\begin{proof}
Let $a_M(0)$ be the constant term of $\mathcal{E}_{k,N}^{(\varepsilon)}(z)$ at the cusp $1/M$. Then by Proposition \ref{bm_constant},  we obtain
\begin{equation}\label{aM0}
	a_M(0)= -\frac{B_k}{2k} \sum\limits_{d\mid N} \varepsilon(d) d^{k/2} \frac{(d,M)^k}{d^k}.
\end{equation}
If $M=1$, we have
\[
	{a_M(0)}= -\frac{B_k}{2k} \varepsilon(N) \left(\frac{1}{N}\right)^{k/2}\sum\limits_{\substack{d\mid N}} \varepsilon\left(\frac{N}{d}\right) \left(\frac{N}{d}\right)^{k/2}= -\frac{B_k}{2k} \varepsilon(N) \left(\frac{1}{N}\right)^{k/2}\sum\limits_{\substack{d\mid N}} \varepsilon(d) d^{k/2}
\] 
which proves the result in this case.\\
If $M=N$, \eqref{aM0} again gives the desired result.\\
Therefore, we only need to consider the cases when $1<M<N$. We can write \eqref{aM0} as
\begin{equation}\label{aM01}
	-\frac{a_M(0)}{\left(\frac{B_k}{2k}\right)}=  \sum\limits_{\substack{d\mid N\\(d,M)=1}} \varepsilon(d) \frac{1}{d^{k/2}} + \sum\limits_{\substack{d\mid N\\(d,M)> 1}} \varepsilon(d) \frac{(d,M)^k}{d^{k/2}}
\end{equation}
and multiplying  by $\varepsilon({\frac{N}{M}})\left(\frac{N}{M}\right)^{k/2}$ on both the  sides gives
\begin{equation}\label{expression_A}
	A= \sum\limits_{\substack{d\mid N\\(d,M)=1}}\varepsilon\left(\frac{N}{M}\right)\varepsilon(d)  \frac{1}{d^{k/2}} \left(\frac{N}{M}\right)^{k/2}  + \sum\limits_{\substack{d\mid N\\(d,M)> 1}}\varepsilon\left(\frac{N}{M}\right)\varepsilon(d)   \frac{(d,M)^k}{d^{k/2}} \left(\frac{N}{M}\right)^{k/2},
\end{equation}
where
\begin{equation}\label{expression_A1}
    A= -\frac{a_M(0)}{\left(\frac{B_k}{2k}\right)}\varepsilon\left({\frac{N}{M}}\right)\left(\frac{N}{M}\right)^{k/2}.
\end{equation}

Let us denote the sums appearing on the right side of \eqref{expression_A} by $S_1$ and $S_2$, respectively. We will calculate these two sums separately.  Because $\varepsilon(d)^2=1$  for any $d\mid N$, we can write 
$$
S_1=\sum\limits_{\substack{d\mid N\\(d,M)=1}}\varepsilon\left(\frac{N}{dM}\right)  \left(\frac{N}{dM}\right)^{k/2}.
$$
Replacing $dM$ by $d'$, we have
\begin{equation}\label{S1}
   S_1=\sum\limits_{\substack{d'\mid N\\M\mid d'}}\varepsilon\left(\frac{N}{d'}\right)  \left(\frac{N}{d'}\right)^{k/2}=\sum\limits_{\substack{d'\mid N\\M\mid d'}}\varepsilon(d')  d'^{k/2}=\sum\limits_{\substack{d\mid N\\M\mid d}}\varepsilon(d)  d^{k/2}. 
\end{equation}
 We now simplify the expression for $S_2$. After rearranging, we may assume that 
 $$M=p_1p_2\dots p_s$$ 
 for some $1\leq s<t$. So for any $d\mid N$ with $(d,M)>1$, we have $(d, p_1\cdots  p_s)= p_{i_1}\cdots p_{i_j}$, where $1\leq i_1< \ldots < i_j\leq s$ for some $j\le s$. Putting $\frac{d}{p_{i_1}\cdots p_{i_j}}=d'$ gives  
  \[
 S_2=	\sum\limits_{j=1}^{s}\sum\limits_{1\leq i_1< \ldots < i_j\leq s}	\sum\limits_{\substack{d'\mid N\\(d',M)=1}}\varepsilon\left(\frac{N}{M}\right) \varepsilon(p_{i_1}\cdots p_{i_j}d')\left( \frac{N}{\frac{d'M}{p_{i_1}\cdots p_{i_j}}}\right)^{k/2}.
 \] 
Since
  \[
 S_2=	\sum\limits_{j=1}^{s}\sum\limits_{1\leq i_1< \ldots < i_j\leq s}	\sum\limits_{\substack{d'\mid N\\(d',M)=1}}\varepsilon\left(\frac{N}{\frac{d'M}{p_{i_1}\cdots p_{i_j}}}\right) \left( \frac{N}{\frac{d'M}{p_{i_1}\cdots p_{i_j}}}\right)^{k/2}.
 \] 
Putting $\frac{d'M}{p_{i_1}\cdots p_{i_j}} = d$, we have
\[ 
 S_2 = \sum\limits_{j=1}^{s}\sum\limits_{1\leq i_1< \ldots < i_j\leq s}	\sum\limits_{\substack{d\mid N\\p_{i_1}\nmid d, \ldots p_{i_j}\nmid d}}\varepsilon\left({\frac{N}{d}}\right)\left( \frac{N}{{d}}\right)^{k/2}= \sum\limits_{j=1}^{s}\sum\limits_{1\leq i_1\leq \ldots \leq i_j\leq s}	\sum\limits_{\substack{d\mid N\\p_{i_1}\nmid d, \ldots p_{i_j}\nmid d}}\varepsilon(d)d^{k/2}.
\]
 Note that the set $\left\{d\mid N: p_{i_1}\nmid d, \ldots ,p_{i_j}\nmid d {\rm{\,\,for\,some \, }} 1\leq i_1< \ldots < i_j\leq s,\, j\in \{1, \ldots s\} \right\}$ is same as the set of all positive divisors $d$ of $N$ such that $M\nmid d$. Therefore 
\begin{equation}\label{S2}
S_2= \sum\limits_{\substack{d\mid N\\M\nmid d}}\varepsilon({ d})   {d}^{k/2}.
  \end{equation}
 Substituting \eqref{expression_A1}, \eqref{S1}, and \eqref{S2} in \eqref{expression_A}, we obtain 
\[
	-\frac{a_M(0)}{\left(\frac{B_k}{2k}\right)}\varepsilon\left({\frac{N}{M}}\right)\left(\frac{N}{M}\right)^{k/2} = \sum\limits_{\substack{d\mid N\\M\mid d}}\varepsilon({ d})   {d}^{k/2}  +  \sum\limits_{\substack{d\mid N\\M\nmid d}}\varepsilon({ d})   {d}^{k/2} = \sum\limits_{\substack{d\mid N}}\varepsilon({ d})   {d}^{k/2}.
\]
Since $ \sum\limits_{\substack{d\mid N}}\varepsilon({ d})   {d}^{k/2} = \prod\limits_{i=1}^{t}(1+\varepsilon(p_i) p_i^{k/2})$, this completes the proof.
\end{proof}

\begin{corollary}\label{mod l cusp form}
  With notation as before, let $\ell\mid \frac{B_k}{2k}\prod\limits_{i=1}^{t} (1+\varepsilon(p_i) p_i^{k/2})$.  Then for any $k\ge 2$, the $\oq{mod\,\ell}$ Eisenstein series $\overline{\mathcal{E}}_{k,N}^{(\varepsilon)}\in \mathcal{S}_k(N,\overline{\mathbb F}_\ell)$.
\end{corollary}

\subsection{The action of $U_p$ operator}
%For each integer $d>0$, $d\mid N$,  the series $E_k(dz)$ is a Hecke eigenform with eigenvalue $\sigma_{k-1}(p)$. Therefore the Eisenstein series $\mathcal{E}_{k,N}^{(\varepsilon)}(z)$ is an eigenform for all the Hecke operators $T_q,$ $q\nmid N$,  with eigenvalues $1+q^{k-1}$.

We now examine the behaviour of the action of $U_p$ operator on   $\mathcal{E}_{k,N}^{(\varepsilon)}(z)$ for $k\ge 2$ and $p\mid N$. We will demonstrate that under certain restrictions on $\ell$, $\oq{mod\,\ell}$ reduction $\overline{\mathcal{E}}_{k,N}^{(\varepsilon)}(z)$  is an eigenfunction for $U_p$ operator.
\begin{proposition}\label{Actn of U_p}
 Let $\mathcal{E}_{k,N}^{(\varepsilon)}(z) = \sum\limits_{n\geq 0} a(n) e^{2 \pi inz}$ and
$
		U_{p} (\mathcal{E}_{k,N}^{(\varepsilon)}(z)) = \sum\limits_{n\geq 0} b(n) e^{2 \pi inz},
$ for a prime $p\mid N$. Then
%Then $b(0)=  -\frac{B_k}{2k} \prod\limits_{i=1}^{s}(1+\varepsilon_i p_i^{k/2})$ and for $n\geq 1$,
%\begin{equation}\label{U_p operator}
%		b (n) =  (1+{p}^{k-1}+\varepsilon(p) p^{k/2})a(n) - (1+\varepsilon(p) p^{k/2})(1+\varepsilon(p) p^{k/2-1}) \sum\limits_{p\mid d,d\mid N} \varepsilon(d) d^{k/2} \sigma_{k-1} \left(\frac{n}{d}\right).
%\end{equation}
\begin{equation}\label{U_p operator}
b(n)= \begin{cases}
    -\frac{B_k}{2k} \prod\limits_{i=1}^{t}(1+\varepsilon(p_i) p_i^{k/2}) & n=0\\
     (1+{p}^{k-1}+\varepsilon(p) p^{k/2})a(n) - (1+\varepsilon(p) p^{k/2})(1+\varepsilon(p) p^{k/2-1})&  n\ge 1. \\
    \hspace{125pt} \times \left(\sum\limits_{p\mid d,d\mid N} \varepsilon(d) d^{k/2} \sigma_{k-1} \left(\frac{n}{d}\right)\right)
\end{cases} 
\end{equation}

\end{proposition}
\begin{proof}
% Recall that 
%$\mathcal{E}_{k,N}^{(\varepsilon)}(z) = \sum\limits_{d\mid N} \varepsilon(d)d^{k/2} E_k(dz) $, 
We  rewrite the expression for $\mathcal{E}_{k,N}^{(\varepsilon)}(z)$  as
\[
     \mathcal{E}_{k,N}^{(\varepsilon)}(z) = \sum\limits_{p\nmid d, d\mid N} \varepsilon(d)d^{k/2} B_dE_k(z) + \sum\limits_{p\mid d, d\mid N} \varepsilon(d)d^{k/2} B_d E_k(z).
    \]
The first sum on the RHS of the above equation is a  linear combination of forms of level $d$ with $p\nmid d$ whereas  the second sum is a form of level divisible by $p$. On applying $U_p$ operator on both sides, we obtain
\begin{equation}\label{U_p exprssn1}
    U_p(\mathcal{E}_{k,N}^{(\varepsilon)}(z)) = \sum\limits_{p\nmid d, d\mid N} \varepsilon(d)d^{k/2} (T_p - p^{k-1}B_p) B_d E_k(z) + \sum\limits_{p\mid d, d\mid N} \varepsilon(d)d^{k/2}U_p B_d E_k(z).
\end{equation}

Let us denote the sums on RHS of the above equation by $S_1$ and $S_2$, respectively. Then 
\[
S_1 = \sum\limits_{p\nmid d, d\mid N} \varepsilon(d)d^{k/2} T_p B_d (E_k(z)) - p^{k-1} \sum\limits_{p\nmid d, d\mid N} \varepsilon(d)d^{k/2}B_p B_d (E_k(z)).
\]
Since $E_k(z)$ is an eigenfunction of the Hecke operator $T_p$ with eigenvalue  $1+p^{k-1}$ and for $(p,d)=1$, the operators $T_p$ and $B_d$ commute, we have
\[
S_1 = \sum\limits_{p\nmid d, d\mid N} \varepsilon(d)d^{k/2} (1+p^{k-1}) B_d  E_k(z)  -  p^{k-1} \sum\limits_{p\nmid d, d\mid N} \varepsilon(d)d^{k/2}E_k(dpz). 
\]

Replacing $dp$ by $d'$ in the second sum of the RHS, we obtain
\begin{equation}\label{U_p 1st sum}
S_1 = (1+p^{k-1}) \sum\limits_{p\nmid d', d'\mid N} \varepsilon(d')d'^{k/2}  E_k(d'z) -  \varepsilon(p) p^{k/2-1} \sum\limits_{p\mid d', d'\mid N} \varepsilon(d')d'^{k/2}E_k(d'z).
\end{equation}

Since the composition $U_p B_p$ becomes identity on $\mathcal{S}_{k}(N)$, we have 
\[
    S_2 = \sum\limits_{p\mid d, d\mid N} \varepsilon(d)d^{k/2} B_{d/p} E_k(z) =  \sum\limits_{p\mid d, d\mid N} \varepsilon(d)d^{k/2}E_k(dz/p).
\]
Replacing $d/p$ by $d'$ yields  
\begin{equation}\label{U_p 2nd sum}
  S_2 =  \varepsilon(p) p^{k/2} \sum\limits_{p\nmid d', d'\mid N} \varepsilon(d')d'^{k/2}E_k(d'z).
\end{equation}

Putting the values of $S_1,S_2$ from \eqref{U_p 1st sum}, \eqref{U_p 2nd sum} respectively in \eqref{U_p exprssn1} and writing $d$ for the dummy variable $d'$ gives
\[
    U_p(\mathcal{E}_{k,N}^{(\varepsilon)}(z)) =  (1+p^{k-1}+\varepsilon(p) p^{k/2})\!\!\!\! \sum\limits_{p\nmid d, d\mid N} \varepsilon(d)d^{k/2}  E_k(dz) -  \varepsilon(p) p^{k/2-1} \sum\limits_{p\mid d, d\mid N} \varepsilon(d)d^{k/2}E_k(dz).
\]

Adding and subtracting $(1+p^{k-1}+\varepsilon(p) p^{k/2}) \sum\limits_{p\mid d, d\mid N} \varepsilon(d)d^{k/2}  E_k(dz)$ on RHS, we get
\begin{align}\label{eigenform for U_p}
U_p(\mathcal{E}_{k,N}^{(\varepsilon)}(z)) = (1+p^{k-1}+\varepsilon(p) p^{k/2}) \mathcal{E}_{k,N}^{(\varepsilon)}(z) & -(1+\varepsilon(p) p^{k/2})(1+\varepsilon(p) p^{k/2-1})\\
&\times \left(\sum\limits_{p\mid d, d\mid N} \!\!\!\!\varepsilon(d)d^{k/2} E_k(dz)\right). \nonumber
    \end{align}
Comparing the Fourier coefficients on both sides, we obtain the desired result.
\end{proof}
%In the above proposition, $ 1+p^{k-1}+\varepsilon(p) p^{k/2}\equiv -\varepsilon(p) p^{k/2-1} {~\rm mod~} \ell$. 
%An immediate consequence of \eqref{eigenform for U_p} yields the following result which plays a crucial role in proving our result in the next section.
 
\begin{corollary}\label{Up eigenfunction}
   Let  $k\ge 2$ and $\ell$ be a prime such that $\ell\mid (1+\varepsilon(p) p^{k/2})(1+\varepsilon(p) p^{k/2-1})$  for a prime $p\mid N$. Then the $\oq{mod\,\ell}$ modular form $\overline{\mathcal{E}}_{k,N}^{(\varepsilon)}(z)$ is an eigenfunction of $U_p$ operator with eigenvalue $-\varepsilon(p) p^{k/2-1}$.
\end{corollary}

\section{Lifting to an eigenform}\label{lift to eigenform}

\begin{theorem}\label{existence of eigenform}
Let $k\geq 2$ be even and $\ell, p_1,p_2,\dots,p_t$ be distinct primes such that $\ell \geq 5$. Let $N=p_1\dots p_t$ and $\varepsilon$ be an Atkin-Lehner eigensystem for $\Gamma_0(N)$.  For $k=2$, we also assume that $\varepsilon(p_i)=-1$ for some  $i$. If 
$$\ell \mid \frac{B_k}{2k}  \prod\limits_{i=1}^{t}(1+\varepsilon(p_i) p_i^{k/2}) \quad {\rm and} \quad \ell\mid (1+\varepsilon(p_i) p_i^{k/2})(1+\varepsilon(p_i) p_i^{k/2-1})~{\mathrm{  for ~each }~1\leq i\leq t}
$$
then there exists an eigenform $f \in \mathcal{S}_k(N)$  and a prime ideal $\Lambda$ over $\ell$ in $K_f$ such that 
$$f(z) \equiv \mathcal{E}_{k,N}^{(\varepsilon)}(z)\modlam.$$
\end{theorem}

\begin{proof}

Under the given assumptions, from Corollary \ref{mod l cusp form} and 
 Corollary \ref{Up eigenfunction}, we see that $$\overline{\mathcal{E}}_{k,N}^{(\varepsilon)}(z)\in \mathcal{S}_k(N,\overline{\mathbb{F}}_\ell), $$ 
is a mod$\,\ell$ eigenfunction of $U_p$ for $p\mid N$.
Hence, ${\mathcal{E}}_{k,N}^{(\varepsilon)}(z)$ is a mod$\,\ell$ eigenform of level $N$ with eigensystem $\varepsilon$. 
From Proposition \ref{carayol},  there exists a form $g\in \mathcal{S}_k(N, \overline{\mathbb Z}_{\ell})$ such that
	\begin{equation}\label{g-E_k}
	    g(z)\equiv \mathcal{E}_{k,N}^{(\varepsilon)}(z)\modl.
	\end{equation}
But $g\in \mathcal{S}_k(N,\overline{\bb{Z}}_{\ell})$ is not necessarily an eigenform however, it is congruent to $\mathcal{E}_{k,N}^{(\varepsilon)}(z)$  modulo $\ell$ which itself is a $\oq{mod\,\ell}$ eigenform. Hence, by Proposition \ref{ds}, there exist an eigenform $h\in \mathcal{S}_k(N,\mathcal{O}_{{K}_h})$ and a prime ideal $\Lambda_h$ over $\ell$ such that the extension $\bb{Q}_{\ell}\subset K_g\subset K_h$ is finite and $h(z)\equiv \mathcal{E}_{k,N}^{(\varepsilon)}(z)\pmod{\Lambda_h}$.  We know that every finite extension of $\bb{Q}_\ell$ is a completion of some number field at a prime lying over $\ell$, hence the eigenform $h$ necessarily arises from $f(z)\in \mathcal{S}_k(N)$ via the embedding of a number field $K\hookrightarrow K_h$ such that
\[
            f(z)\equiv \mathcal{E}_{k,N}^{(\varepsilon)}(z)\modlam.
\]
where $\Lambda$ is a prime ideal over $\ell$ in $K$.
\end{proof}

 \section{Proof of Theorem \ref{newform_dp}}\label{proof_main}
 
From the assumptions on $\ell$, we observe that for weight $k\ge 2$,
$$\ell \mid \frac{B_k}{2k}  \prod\limits_{q\mid Np}(1+\varepsilon(q) q^{k/2}) \quad {\rm and} \quad \ell\mid (1+\varepsilon(q) q^{k/2})(1+\varepsilon(q) q^{k/2-1})~{\mathrm{  for ~each }~q\mid Np},
$$
meaning that $\ell$ satisfies the hypothesis of Theorem \ref{existence of eigenform}. Therefore, there exists a Hecke eigenform $g\in \mathcal{S}_k(Np)$  such that
\begin{equation}\label{eq:8}
		g(z)\equiv \mathcal{E}_{k,Np}^{(\varepsilon)}(z) \modlam,
\end{equation}
where $\Lambda$ is a prime ideal
above $\ell$ in the compositum of coefficient fields of all normalized eigenforms in $\mathcal{S}_k(d)$, for $d\mid Np$. From Chebotarev density theorem, it is evident that 
\begin{equation}\label{iso}
   {\overline\rho_{g,\Lambda}}\simeq 1\oplus \overline \chi_\ell^{k-1}.
\end{equation}
We remark that for $k=2$, our assumptions compel us to take $\varepsilon(p)=1$ and $\varepsilon(q)=-1$.
Now we consider the following two cases:\\
{\bf Case (i):  $g$ is a newform of level $Np$.}   
 We show that if $\delta$ is the Atkin-Lehner eigensystem of $g$, then $\delta=\varepsilon$. For any prime  $q\mid Np$, we know that $a_g(q) = -\delta(q) q^{k/2-1}$. Comparing $q$th Fourier coefficients on both the sides in \eqref{eq:8} yields 
\[
      -\delta(q) q^{k/2-1}\equiv 1+q^{k-1}+\varepsilon(q) q^{k/2}\equiv -\varepsilon(q) q^{k/2-1} \modlam. 
\]
Hence, $\ell\mid (\delta(q)-\varepsilon(q))q^{k/2-1}$. But $\ell\nmid q$, so we have $\delta(q)=\varepsilon(q)$ and $g\in \mathcal{S}_k^{(\varepsilon)}(Np)$ is the desired newform.\\
{\bf Case (ii): $g$ is not a newform of level $Np$.} We first claim that $g$ can not arise from a newform of level dividing $N$.
On the contrary, assume that $g$ arises from a newform $h\in \mathcal{S}_k(M)$ for some $M\mid N$.
%\textcolor{red}{The case $M=1$ needs to be considered separately.}\\

We first consider the case when $M=1$. We need to assume that $k\ge 12$ and $k\neq 14$ because there are no newforms of other weights and of level $1$.
 Let $h$ be a newform of level $1$ such that   
\begin{equation}\label{gh level 1}
     g(z) = h(z) + \sum\limits_{1<d\mid Np} \alpha(d) h(dz),
\end{equation}
   where $\alpha(d)\in \mathbb{C}$ for each $d$. Thus  \eqref{iso} gives that  the $\oq{mod}\,\ell$ Galois representation ${\overline{\rho}_{h,\Lambda}}$ is isomorphic to $1\oplus \overline{\chi}_{\ell}^{k-1}$. Both the representations are unramified outside $ \ell$, so
   \[
        a_h(q) \equiv 1+q^{k-1}\modlam, {\rm ~for ~any~prime~} q\mid Np. 
   \]
From \eqref{gh level 1} and \eqref{eq:8},  we obtain  $a_h(q) = a_g(q) \equiv 1 + q^{k-1}\modlam$ for $q\nmid Np$. Hence, from Proposition \ref{lang} we deduce that
\[
    h(z) \equiv E_{k}(z)\modlam,
\] inferring that $\ell\mid \frac{B_k}{2k}$, which is a contradiction.

For $M >1$, we have
%after rearrangement, we may assume that $M=p_1\ldots p_t$ for some $1\leq t\leq s$. Therefore,
\begin{equation}\label{gh}
		g(z) = h(z) + \sum\limits_{1<d\mid \frac{Np}{M}} \alpha(d) h(dz),
\end{equation}
where $\alpha(d)\in \mathbb{C}$. It is clear that $
a_h(q)=a_g(q)$ for any prime $q\mid M\ell$. So \eqref{eq:8} gives
\begin{align*}
    a_h(q)&\equiv 1+q^{k-1}+\varepsilon(q)q^{k/2}\modlam {\rm ~for ~each~prime~} q\mid M;~ {\rm and}\\
    a_h(\ell)&\equiv 1+\ell^{k-1} \modlam.
\end{align*}
Furthermore, because $\overline{\rho}_{h,\Lambda} \simeq 1\oplus \overline \chi_\ell^{k-1}$ and both these representations are unramified outside $M\ell$, we have
\[
		a_h(q)\equiv 1+q^{k-1}\modlam, {\rm ~for ~any~prime~} q\nmid M\ell.
\]
Combining the last three congruences and applying Proposition \ref{lang}, we conclude that
\[
		h(z) \equiv \mathcal{E}_{k,M}^{(\varepsilon)}(z)\modlam,
\]
implying that $\ell \mid  \frac{B_k}{2k} \prod_{q\mid M}(1+\varepsilon(q) q^{k/2})$. Therefore,  $\ell\mid (1+\varepsilon(q) q^{k/2})$ for some $q\mid M$. But we know that $\ell\mid (1+\varepsilon(q) q^{k/2-1})$ which gives $\ell\mid \varepsilon(q) q^{k/2-1}(1-q)$. This is not possible because $\ell\nmid \phi(N)$. \\
 Hence, $g$ must arise from a newform of level $p$ or of level $dp$ for some proper divisor $d$ of $N$.\\
 {\bf Subcase (a): The form $g$ arises from a newform $f\in \mathcal S_k^{(\delta)}(dp)$, for some $1<d\mid N$.} Clearly, we have $\overline{\rho}_{f,\Lambda} \simeq 1\oplus \overline \chi_\ell^{k-1}$. As proved in {\bf Case (i)}, we can easily show that $\delta(q)=\varepsilon(q)$, for any $q\mid dp$. Considering the congruences between Fourier coefficients, we have
 $$f(z) \equiv \mathcal{E}_{k,dp}^{(\varepsilon)}(z)\modlam$$ 
 which completes the proof in this case.\\
  {\bf Subcase (b): The form $g$ arises from a newform $h\in \mathcal S_k(p)$.} In this case, we have
 \begin{equation}\label{gh2}
     g(z) = h(z) + \sum_{1<d\mid N} \alpha(d) h(dz),
 \end{equation}
 for some $\alpha(d)\in \bb{C}$ giving
 $\overline{\rho}_{h,\Lambda} \simeq 1\oplus \overline \chi_\ell^{k-1}$.
 Since both representations are  unramified outside $p\ell$, we have 
$$a_h(p_{r})\equiv 1+p_{r}^{k-1} \modlam.$$ 
The assumption $\ell\mid (1+\varepsilon(p_r) p_r^{k/2-1})$ gives 
\[
     a_h(p_r)\equiv -\varepsilon(p_r) p_r^{k/2-1}(1+p_r)  \modlam.
\]
    Thus, the newform $h(z)$ satisfies the level-raising condition at $p_r$. Applying Theorem \ref{level raising theorem}, 
    we obtain a newform $f$ either in $\mathcal{S}^{(\delta)}_k(p_r)$ or in $\mathcal{S}^{(\delta)}_k(pp_r)$, where $\delta$ denotes an Atkin-Lehner eigensystem of $f$ of level $p_r$ or $pp_r$ such that
    \begin{equation}\label{gh3}
    a_f(q)\equiv a_h(q) \modlam  \mathrm{~for ~all} ~q\nmid p_rp,
    \end{equation}
and in particular,
\begin{equation}\label{gh4}
{\overline\rho_{f,\Lambda}}\simeq 1\oplus \overline \chi_\ell^{k-1}.
\end{equation}
First, we will show that $f$ can not be a newform of level $p_r$. This is clearly true if $k=2$ using \cite[Proposition 5.12]{maz} because $\ell \nmid (p_r-1)$ which is also mentioned in Remark \ref{k=2}. This completes the proof of the assertion stated in Remark \ref{k=2}. So assume that $k\ge 4$. If possible, let   $f\in \mathcal{S}^{(\delta)}_k(p_r)$ be a newform. %hence $a_f(p_j) = -\delta(p_j)p_j^{k/2-1}$. %Let $\delta(p_j) = -\varepsilon(p_j)$ so that $a_f(p_j) = \varepsilon(p_j)p_j^{k/2-1}.$
Applying Theorem \ref{langland} and using \eqref{gh4}, we obtain
\[
            \oq{tr}(\overline{\rho}_f|^{}_{D_{p_r}}(\oq{Frob}_{p_r}))\equiv (p_r^{k/2}+p_r^{k/2-1})\frac{a_f(p_r)}{p_r^{k/2-1}}  \equiv 1+p_r^{k-1}\modlam
    \]
    which gives 
    $$
    -\delta(p_r)p_r^{k/2-1}(1+p_r) \equiv -\varepsilon(p_r)p_r^{k/2-1}(1+p_r)\modlam.
    $$
Therefore, $\ell\mid \left( \delta(p_r)-\varepsilon(p_r)\right) p_r^{k/2-1}(1+p_r)$. %We have the following two cases.
%{\bf{Case (i): $\ell\nmid (p_j+1)$}.}
%so that $  a_f(p_j) = -\delta(p_j) p_j^{k/2-1}$.
Since $\ell\nmid p_r(1+p_r),$ we have $\delta(p_r)=\varepsilon(p_r)$ so that
$$
a_f(p_r)\equiv 1+p_r^{k-1}+\varepsilon(p_r)p_r^{k/2}\modlam.
$$
Also, from \eqref{gh3} and \eqref{gh4}, we obtain
    $$a_f(q)\equiv 1+q^{k-1}\modlam \mathrm {~for ~all~} q\nmid p_r.
    $$
Combining the last two congruences and invoking Theorem \ref{lang} gives
 $f(z)\equiv \mathcal{E}_{k,p_r}^{(\varepsilon)}(z)\modlam$ and comparing the constant term, we obtain
 $$    \ell \mid (1+\varepsilon(p_r)p_r^{k/2}).$$
  But  $\ell \mid (1+\varepsilon(p_r)p_r^{k/2-1})$, therefore  $\ell\mid p_r(p_r-1)$ which is a contradiction.

%{\bf{Case (ii): $\ell\mid (p_j+1)$}.} We again claim that $\delta(p_j) = \varepsilon(p_j)$. For if $\delta(p_j) = -\varepsilon(p_j)$, then $  a_f(p_j) = -\delta(p_j) p_j^{k/2-1}= \varepsilon(p_j) p_j^{k/2-1}$. , we get 
 Summarising the above discussion, we have a newform $f\in \mathcal{S}^{(\delta)}_k(pp_r)$ satisfying \eqref{gh3} and \eqref{gh4}. Applying Theorem \ref{langland} and following the same argument as before,  the fact that $\ell\nmid (1+p_r)$ gives $\delta(p_r)=\varepsilon(p_r)$.

 Next, we claim that $\delta(p) = \varepsilon(p)$. On the contrary, let us assume that $\delta(p) =- \varepsilon(p)$. In this case $a_f(p)  = \varepsilon(p)p^{k/2-1}$. Therefore using  Theorem \ref{langland} and  \eqref{gh4}, we obtain
 $$
        \left(p^{k/2}+p^{k/2-1}\right) \frac{\varepsilon(p)p^{k/2-1}}{p^{k/2-1}} \equiv 1+p^{k-1} \modlam,
 $$ which in turn implies
  $1+p^{k-1}-\varepsilon(p)p^{k/2}-\varepsilon(p)p^{k/2-1}\equiv 0\modl$, and consequently 
$$
        a_f(p) = \varepsilon(p)p^{k/2-1}\equiv 1+p^{k-1}-\varepsilon(p)p^{k/2-1}\modlam.
$$
Using \eqref{gh3} and \eqref{gh4}, we obtain that $f(z)\equiv \mathcal{E}_{k,pp_r}^{(\delta)}(z)\modlam$. Again applying Theorem \ref{lang} gives 
$$
            \ell\mid \frac{B_k}{2k}\left(1+\varepsilon(p_r)p_r^{k/2}\right)\left(1-\varepsilon(p)p^{k/2}\right),
$$
which is a contradiction to our assumptions on $\ell$. Hence $\delta(p) = \varepsilon(p)$.\\
Finally, by considering congruences among Fourier coefficients, we obtain
 $$f(z) \equiv \mathcal{E}_{k,pp_r}^{(\varepsilon)}(z)\modlam$$
which completes the proof.

 % \subsection{Proof of Theorem \ref{result for N=p}}

%Since $\ell$ satisfies the hypothesis of Theorem \ref{existence of eigenform}, there exists a Hecke eigenform $g(z)\in \mathcal{S}_k(p_1p_2)$ and a prime ideal $\Lambda$ lying over $\ell$ in the coefficient field $K_g$  such that
%\begin{equation}
%		g\equiv \mathcal{E}_{k,p_1p_2}^{(\varepsilon)}\, \,\mathrm{mod}\, \Lambda.
%\end{equation}

 \section{Proof of Theorem \ref{density}}
Define the sets of primes $\mathcal{P}$ and $\mathcal{Q}$ as follows:
\[
\mathcal{P} = \{p {\rm~prime}:  ~p\equiv -1\modl\} \qquad \mathrm{and}\quad \mathcal{Q} = \{q {\rm~prime}:  ~q\not\equiv \pm{1}\modl\}.
\]
By Dirichlet's theorem on primes in arithmetic progressions, the densities of $\mathcal P$ and  $\mathcal Q$ are   $ \frac{1}{\ell}$ and $\frac{\ell-2}{\ell},$ respectively. 
Since  $k = 2$ or $\ell+1$, we can take $k = n(\ell-1)+2$, where $n=0$ or $1$. By Kummer's congruence, we have $\frac{B_{k}}{k}\equiv \frac{B_2}{2}\modl$ and thus $   \ell\nmid \frac{B_{k}}{2k}.$

For any prime $p\in \mathcal{P}$ and  $q\in \mathcal{Q}$, the following congruences hold.
\begin{align*}
    p^{k/2}&\equiv \begin{cases}
        -1\modl &{\rm ~~if~} \ell\equiv 1\pmod{4}\\
        (-1)^{n+1} \modl &{\rm ~~if~~} \ell\equiv 3\pmod{4};
    \end{cases} \\ 
     q^{k/2-1} &\equiv \begin{cases}
    ~1 \modl & \quad {\rm ~~if~~} q {\rm~ is~a~quadratic ~residue~modulo~} \ell\\
    (-1)^{n}\modl &\quad  {~\rm otherwise}.
    \end{cases}
\end{align*}
Let $\varepsilon$ be the Atkin-Lehner eigensystem $\varepsilon$ for level $pq$ defined by 
$$ \varepsilon(p)=-p^{k/2}\modl \quad {\rm and} \quad \varepsilon(q) = -q^{k/2-1}\modl.
$$
For any $p\in \mathcal{P}$ and $q\in \mathcal{Q}$,  we have the following
\begin{equation*}\label{condition on p}
    \ell\nmid \frac{B_k}{2k}(p-1) (q^2-1), \quad  \ell\mid (1+\varepsilon(p)p^{k/2}), \quad \mathrm{and} \quad \ell\mid (1+\varepsilon(q)q^{k/2-1}).
\end{equation*}

By Theorem \ref{result for N=p}, there exists a newform $f\in \mathcal{S}_{k}^{(\varepsilon)}(pq)$  and a prime ideal $\Lambda$ over $\ell$ in a sufficiently large number field such that
\[
    f(z)\equiv \mathcal{E}_{k     ,pq}^{(\varepsilon)}(z)\modlam. 
\]

    \section{Proof of Lemma \ref{Converse}}\label{proof of converse}
We prove the direct implication of the lemma. Let $a_f(n)$ and $a(n)$ be the $n$th Fourier coefficients of the newform $f\in \mathcal{S}_k^{(\varepsilon)}(N)$ and $\mathcal{E}_{k,N}^{(\varepsilon)}$, respectively.  Since  $\bar{\rho}_{f,\Lambda}\simeq 1\oplus \overline{\chi}_{\ell}^{k-1}$, for any $q\nmid N\ell$
\begin{equation}\label{unramified at q}
    a_f(q)\equiv 1+q^{k-1}\modlam.
\end{equation}

For $p\mid N$, using Theorem \ref{langland}, 
we obtain 
\[
        \oq{tr}\, \left(\bar{\rho}_f|^{}_{D_p}(\oq{Frob}_p)\right) = -\varepsilon(p)p^{k/2-1}(1+p)\equiv 1+p^{k-1} \modlam,
\]

which proves that  
\begin{equation}\label{second condition of conjecture}
    \ell\mid (1+\varepsilon(p)p^{k/2})(1+\varepsilon(p)p^{k/2-1}).
\end{equation}
Also for  $p\mid N$, $a(p) = 1+p^{k-1}+\varepsilon(p)p^{k/2}$ and $a_f(p) = -\varepsilon(p)p^{k/2-1}$, hence
\begin{equation*}
    a_f(p) \equiv a(p) \modlam.
\end{equation*}
In view of \eqref{second condition of conjecture} and Corollary \ref{Up eigenfunction}, ${\mathcal{E}}_{k,N}^{(\varepsilon)}$ is a $\oq{mod}\,\ell$ eigenform. Therefore, combining \eqref{unramified at q} with the previous congruence gives
\begin{equation}\label{final congruence}
     a_f(n) \equiv a(n) \modlam.
\end{equation}
for each  $n$ with $(\ell,n)=1$.
Denote the theta operator by  $\Theta \left(=\frac{1}{2\pi i}\frac{d}{dz}\right)$. Then the congruence relation \eqref{final congruence} backs us to write
\[
        \Theta({f})(z) = \Theta({\mathcal{E}}_{k,N}^{(\varepsilon)})(z) \modlam.
\]

As $\ell >k+1$, appealing a result of Katz on $\Theta$ operator \cite{katz2} stating that $\Theta$ is injective, we obtain
\[
    {f}(z) \equiv {\mathcal{E}}_{k,N}^{(\varepsilon)}(z) \modlam.
\]
%The reverse implication holds as a consequence of Chebotarev density theorem.

\section{Examples}\label{examples}
We now give some numerical examples to demonstrate our results. For simplicity, we write ${\mathfrak q}$ for $e^{2\pi i z}$. We recall Sturm's bound that states: two normalized eigenforms $f(z) = \sum_{n\ge 1}a_f(n){\mathfrak q}^n$ and $g(z) = \sum_{n\ge 1}a_g(n){\mathfrak q}^n \in \mathcal{S}_k(N)$  are congruent modulo a prime ideal $\Lambda$ if $a_f(n)\equiv a_g(n)\modlam$ for all $n\le \left\lfloor \frac{k}{12} \prod_{p\mid N}\left(  1+ \frac{1}{p} \right)\right\rfloor$. We use this bound to verify the congruences in the following examples.   
The computations involved are performed on Sage.

%Theorem \ref{result for N=p} using \cite{sage}. 

%\begin{example}
%For weight $k=2$, take  $\ell =5$, $p=19$, and $q=2$. We easily see that $\ell\nmid \frac{B_2}{2\cdot 2}(19-1)(2^2-1)$.  Set $\varepsilon(19)=1$ and $\varepsilon(2)=-1$. By Theorem \ref{result for N=p}, there exists a newform $f\in \mathcal{S}_2^{(\varepsilon)}(pq)$ satisfying the relation \eqref{cong_pq}.
%Consider the newform $f\in \mathcal{S}_2^{(\varepsilon)}(38)$  given by $$f(z) = q + q^2 - q^3 + q^4 - 4q^5 - q^6 + 3q^7 + q^8 - 2q^9 + O(q^{10}).$$

%The coefficient field of $f$ is the rational field and one can check that
%\[
% f(z) \equiv \mathcal{E}_{2,38}^{(\varepsilon)}(z) ~\oq{mod}\,5.
% \]
%It is evident from Sturm bound that we need to check the congruence of the first $10$ Fourier coefficients.
    
%\end{example}

\begin{example}
    Take $k=6$, $\ell=5$, $p=19$, and $q=3$. We see that $5\nmid\frac{B_6}{2\cdot 6}\phi(19\cdot 3)(3+1)$. For the Atkin-Lehner eigensystem $\varepsilon$ of $\Gamma_0(57)$ given by $\varepsilon(19)=\varepsilon(3)=1$, we have 
    \[
            5\mid (1+\varepsilon(19)19^3) \quad \oq{and}\quad 5\mid (1+\varepsilon(3)3^2).
    \]
      The hypotheses of Theorem \ref{result for N=p} are satisfied, therefore   existence of a newform $f\in \mathcal{S}_k^{(\varepsilon)}(57)$  and a prime ideal $\Lambda$ lying over $5$ satisfying $f(z)\equiv \mathcal{E}_{6,57}^{(\varepsilon)}(z) \modlam$ is guaranteed. When performing computations on Sage, we obtain that 
$$
        f(z) =   {\mathfrak q} + a{\mathfrak q}^2 - 9{\mathfrak q}^3 + (a^2 - 32){\mathfrak q}^4 + \left(-{a^3}/{6}  - {19a^2}/{6}  + {14a}/{3}  + {410}/{3}\right){\mathfrak q}^5 - 9 a {\mathfrak q}^6 + O({\mathfrak q}^7)
        % (\frac{5}{6} a_5^3 + \frac{35}{6} a5^2 - \frac{178}{5} a_5 - \frac{802}{3}){\mathfrak q}^7 + (a5^3 - 64a5){\mathfrak q}^8 + 81{\mathfrak q}^9 + O({\mathfrak q}^{10}) 
        \in \mathcal{S}_6^{(\varepsilon)}(57)
 $$
  where $a$ is a root of the polynomial $x^4 - x^3 - 90x^2 + 118x + 1412$ and $\Lambda = (5, a+2)$ are the desired newform and prime ideal respectively.    
\end{example}

\begin{example}\label{exm}
    Take $k=2$, $\ell=5$, $p = 19$ and $N=6$, then $5\nmid \phi(6\cdot 19)(2+1)(3+1)$. Taking $\varepsilon(19) = 1$, $\varepsilon(2)=-1$ and $\varepsilon(3)=-1$, we see that assumptions of Theorem $\ref{newform_dp}$ are satisfied. Therefore, there must exist a newform in $\mathcal{S}_2^{(\varepsilon)}(19 d)$, where $1<d\mid 6$, which is congruent to $\mathcal{E}_{2,19d}^{(\varepsilon)}(z)
    $ modulo some prime above the prime 5. Using Sage, we have checked that this is true for $d=2$ and also for $d=3$. More precisely, there are newforms 
        \begin{align*}
f(z)&= {\mathfrak q} + {\mathfrak q}^2 - {\mathfrak q}^3 + {\mathfrak q}^4 - 4  {\mathfrak q}^5 - {\mathfrak q}^6 + 3  {\mathfrak q}^7 + {\mathfrak q}^8 - 2  {\mathfrak q}^9 + O({\mathfrak q}^{10}) \in \mathcal{S}_2^{(\varepsilon)}(38); \quad {\rm and}\\
g(z)&= {\mathfrak q} - 2{\mathfrak q}^2 + {\mathfrak q}^3 + 2{\mathfrak q}^4 + {\mathfrak q}^5 - 2{\mathfrak q}^6 + 3{\mathfrak q}^7 + {\mathfrak q}^9 + O({\mathfrak q}^{10}) \in \mathcal{S}_2^{(\varepsilon)}(57)
    \end{align*}
    satisfying $f(z) \equiv \mathcal{E}_{2,38}^{(\varepsilon)}(z)\pmod 5$ and $g(z) \equiv \mathcal{E}_{2,57}^{(\varepsilon)}(z)\pmod 5$.\\
    Furthermore, for level 114 (the case $d=6$), even though the divisibility conditions in Conjecture \ref{conjecture} are satisfied, there does not exist a newform in $\mathcal{S}_2^{(\varepsilon)}(114)$ with Atkin-Lehner eigensystem $\varepsilon$. Indeed, there is no newform in $\mathcal{S}_2(114)$ with reducible mod 5 representation and so Conjecture \ref{conjecture} is not true for $k=2$, in general.

\end{example}

\begin{example}
Take $k=6$, $\ell=13$, $p=3$, and $N = 5\cdot 31$. We see that $13\nmid \phi(3\cdot 5\cdot 31)(5+1)(31+1)$ and if we set $\varepsilon(3)=-1$, $\varepsilon(5)=1$, and $\varepsilon(31)=1$, we have
\[
    13\mid(1+\varepsilon(3)^3), \quad 13\mid (1+\varepsilon(5)5^2), \quad \text{and}\quad 13\mid (1+\varepsilon(31)31^2).
\]

Thus the assumptions of Theorem \ref{newform_dp} are satisfied, hence there must exist a newform in $\mathcal{S}_6^{(\varepsilon)}(3d)$, for $1< d\mid 155$, which is congruent to $\mathcal{E}_{6,3d}^{(\varepsilon)}$ modulo some prime ideal above the prime $13$.  Performing computations on Sage, we have verified that the newforms 
\begin{align*}
f(z) &= {\mathfrak q} + 7  {\mathfrak q}^2 + 9  {\mathfrak q}^3 + 17  {\mathfrak q}^4 -25 {\mathfrak q}^5 + O({\mathfrak q}^6)  \in \mathcal{S}_{6}^{(\varepsilon)}(15)
\quad \rm{and} \\ 
 g(z) &= {\mathfrak q} + a {\mathfrak q}^2 + 9{\mathfrak q}^3 + (a^2 - 32){\mathfrak q}^4 + O({\mathfrak q}^5) \in \mathcal{S}_{6}^{(\varepsilon)}(93),
\end{align*}
where $a$ is a root of polynomial  $x^8 - 9x^7 - 184x^6 + 1479x^5 + 10247x^4 - 65022x^3 - 172008x^2 + 414408x + 896048$,
satisfy $f(z)\equiv \mathcal{E}_{6,15}^{(\varepsilon)}(z)\pmod{13} $ and $g(z)\equiv \mathcal{E}_{6,93}^{(\varepsilon)}(z)\pmod{\Lambda}$ for the prime\\
$\Lambda =  (13, 303379/33108088 a^7 - 2359769/33108088 a^6 
+\dots 
%- 24762839/16554044 a^5 + 351171713/33108088 a^4 + 1950954491/33108088a^3 - 1698301795/4138511 a^2 + 263028539/4138511 a 
+ 8389270401/4138511)$.
%The defining polynomial of coefficient field is $x^8 - 9*x^7 - 184*x^6 + 1479*x^5 + 10247*x^4 - 65022*x^3 - 172008*x^2 + 414408*x + 896048$ 

Moreover, the divisibility conditions of Conjecture $\ref{conjecture}$ are also satisfied. Using Sage, we check that for the newform $f\in \mathcal{S}_6^{(\varepsilon)}(465)$ given by 
\[
        f(z) = {\mathfrak q} + a {\mathfrak q}^2 + 9 {\mathfrak q}^3 + (a^2 - 32) {\mathfrak q}^4 - 25 {\mathfrak q}^5 + 9 a {\mathfrak q}^6 + O({\mathfrak q}^7),
\]
where $a$ is a root of the polynomial $x^{13} - 7x^{12} - 290x^{11} + 1776x^{10} + \cdots - 866822400$ and a prime ideal $\Lambda = (13,a-7)$ over $13$, we have
\[
        f(z) \equiv \mathcal{E}_{6,465}^{(\varepsilon)}(z) \modlam.
\]

%$x^{13} - 7x^{12} - 290x^{11} + 1776x^{10} + 31321x^9 - 159677x^8 - 1552236x^7 + 6128528x^6 + 35045264x^5 - 99418400x^4 - 310314176x^3 + 612277504x^2 + 851704832x - 866822400
%$

\end{example}

	\bibliography{output.bbl}

\begin{bibdiv}
\begin{biblist}

\bib{al}{article}{
      author={Atkin, A. O.~L.},
      author={Lehner, J.},
       title={Hecke operators on {$\Gamma \sb{0}(m)$}},
        date={1970},
        ISSN={0025-5831,1432-1807},
     journal={Math. Ann.},
      volume={185},
       pages={134\ndash 160},
         url={https://doi.org/10.1007/BF01359701},
      review={\MR{268123}},
}

\bib{bm}{article}{
      author={Billerey, N.},
      author={Menares, R.},
       title={On the modularity of reducible {${\rm mod}\, \ell$} {G}alois
  representations},
        date={2016},
        ISSN={1073-2780,1945-001X},
     journal={Math. Res. Lett.},
      volume={23},
      number={1},
       pages={15\ndash 41},
         url={https://doi.org/10.4310/MRL.2016.v23.n1.a2},
      review={\MR{3512875}},
}

\bib{bm2}{article}{
      author={Billerey, N.},
      author={Menares, R.},
       title={Strong modularity of reducible {G}alois representations},
        date={2018},
        ISSN={0002-9947,1088-6850},
     journal={Trans. Amer. Math. Soc.},
      volume={370},
      number={2},
       pages={967\ndash 986},
         url={https://doi.org/10.1090/tran/6979},
      review={\MR{3729493}},
}

\bib{bet}{article}{
      author={Bettin, S.},
      author={Perret-Gentil, C.},
      author={Radziwill, M.},
       title={A note on the dimension of the largest simple {H}ecke submodule},
        date={2021},
        ISSN={1073-7928,1687-0247},
     journal={Int. Math. Res. Not. IMRN},
      number={7},
       pages={4907\ndash 4919},
         url={https://doi.org/10.1093/imrn/rny287},
      review={\MR{4241117}},
}

\bib{svd}{misc}{
      author={Deo, S.~V.},
       title={Non-optimal levels of some reducible mod $p$ modular
  representations},
        date={2024},
}

\bib{df}{article}{
      author={Dummigan, N.},
      author={Fretwell, D.},
       title={Ramanujan-style congruences of local origin},
        date={2014},
        ISSN={0022-314X,1096-1658},
     journal={J. Number Theory},
      volume={143},
       pages={248\ndash 261},
         url={https://doi.org/10.1016/j.jnt.2014.04.008},
      review={\MR{3227346}},
}

\bib{dia}{incollection}{
      author={Diamond, F.},
       title={Congruence primes for cusp forms of weight {$k\ge 2$}},
        date={1991},
       pages={6, 205\ndash 213},
        note={Courbes modulaires et courbes de Shimura (Orsay, 1987/1988)},
      review={\MR{1141459}},
}

\bib{dur}{article}{
      author={Dieulefait, L.~V.},
      author={Jim\'{e}nez~Urroz, J.},
      author={Ribet, K.~A.},
       title={Modular forms with large coefficient fields via congruences},
        date={2015},
        ISSN={2522-0160,2363-9555},
     journal={Res. Number Theory},
      volume={1},
       pages={Paper No. 2, 14},
         url={https://doi.org/10.1007/s40993-015-0003-9},
      review={\MR{3500986}},
}

\bib{diashur}{book}{
      author={Diamond, F.},
      author={Shurman, J.},
       title={A first course in modular forms},
      series={Graduate Texts in Mathematics},
   publisher={Springer-Verlag, New York},
        date={2005},
      volume={228},
        ISBN={0-387-23229-X},
      review={\MR{2112196}},
}

\bib{dl}{article}{
      author={Diamond, F.},
      author={Taylor, R.},
       title={Nonoptimal levels of mod {$l$} modular representations},
        date={1994},
        ISSN={0020-9910,1432-1297},
     journal={Invent. Math.},
      volume={115},
      number={3},
       pages={435\ndash 462},
         url={https://doi.org/10.1007/BF01231768},
      review={\MR{1262939}},
}

\bib{gp}{article}{
      author={Gaba, R.},
      author={Popa, A.},
       title={A generalization of {R}amanujan's congruence to modular forms of
  prime level},
        date={2018},
        ISSN={0022-314X,1096-1658},
     journal={J. Number Theory},
      volume={193},
       pages={48\ndash 73},
         url={https://doi.org/10.1016/j.jnt.2018.05.007},
      review={\MR{3846800}},
}

\bib{katz}{incollection}{
      author={Katz, N.~M.},
       title={{$p$}-adic properties of modular schemes and modular forms},
        date={1973},
   booktitle={Modular functions of one variable, {III} ({P}roc. {I}nternat.
  {S}ummer {S}chool, {U}niv. {A}ntwerp, {A}ntwerp, 1972)},
      series={Lecture Notes in Math.},
      volume={Vol. 350},
   publisher={Springer, Berlin-New York},
       pages={69\ndash 190},
      review={\MR{447119}},
}

\bib{katz2}{incollection}{
      author={Katz, N.~M.},
       title={A result on modular forms in characteristic {$p$}},
        date={1977},
   booktitle={Modular functions of one variable, {V} ({P}roc. {S}econd
  {I}nternat. {C}onf., {U}niv. {B}onn, {B}onn, 1976)},
      series={Lecture Notes in Math.},
      volume={Vol. 601},
   publisher={Springer, Berlin-New York},
       pages={53\ndash 61},
      review={\MR{463169}},
}

\bib{kkms}{article}{
      author={Kumar, A.},
      author={Kumari, M.},
      author={Moree, P.},
      author={Singh, S.~K.},
       title={Ramanujan-style congruences for prime level},
        date={2023},
        ISSN={0025-5874,1432-1823},
     journal={Math. Z.},
      volume={303},
      number={1},
       pages={Paper No. 19, 22},
         url={https://doi.org/10.1007/s00209-022-03159-5},
      review={\MR{4521621}},
}

\bib{lan}{incollection}{
      author={Langlands, R.~P.},
       title={Modular forms and {$\ell $}-adic representations},
        date={1973},
   booktitle={Modular functions of one variable, {II} ({P}roc. {I}nternat.
  {S}ummer {S}chool, {U}niv. {A}ntwerp, {A}ntwerp, 1972)},
      series={Lecture Notes in Math.},
      volume={Vol. 349},
   publisher={Springer, Berlin-New York},
       pages={361\ndash 500},
      review={\MR{354617}},
}

\bib{lang1}{book}{
      author={Lang, S.},
       title={Introduction to modular forms},
      series={Grundlehren der mathematischen Wissenschaften [Fundamental
  Principles of Mathematical Sciences]},
   publisher={Springer-Verlag, Berlin},
        date={1995},
      volume={222},
        ISBN={3-540-07833-9},
        note={With appendixes by D. Zagier and Walter Feit, Corrected reprint
  of the 1976 original},
      review={\MR{1363488}},
}

\bib{LMPM}{article}{
      author={Luca, F.},
      author={Menares, R.},
      author={Pizarro-Madariaga, A.},
       title={On shifted primes with large prime factors and their products},
        date={2015},
        ISSN={1370-1444,2034-1970},
     journal={Bull. Belg. Math. Soc. Simon Stevin},
      volume={22},
      number={1},
       pages={39\ndash 47},
         url={http://projecteuclid.org/euclid.bbms/1426856856},
      review={\MR{3325719}},
}

\bib{maz}{article}{
      author={Mazur, B.},
       title={Modular curves and the eisenstein ideal},
        date={1977},
     journal={Inst. Hautes \'{E}tudes Sci. Publ. Math.},
      volume={47},
       pages={33\ndash 186},
         url={https://doi.org/10.1007/BF02684339},
}

\bib{ogg}{article}{
      author={Ogg, A.~P.},
       title={On the eigenvalues of {H}ecke operators},
        date={1969},
        ISSN={0025-5831,1432-1807},
     journal={Math. Ann.},
      volume={179},
       pages={101\ndash 108},
         url={https://doi.org/10.1007/BF01350121},
      review={\MR{269597}},
}

\bib{rib2}{article}{
      author={Ribet, K.~A.},
       title={On {$l$}-adic representations attached to modular forms},
        date={1975},
        ISSN={0020-9910,1432-1297},
     journal={Invent. Math.},
      volume={28},
       pages={245\ndash 275},
         url={https://doi.org/10.1007/BF01425561},
      review={\MR{419358}},
}

\bib{hpf}{incollection}{
      author={Swinnerton-Dyer, H. P.~F.},
       title={On {$l$}-adic representations and congruences for coefficients of
  modular forms},
        date={1973},
   booktitle={Modular functions of one variable, {III} ({P}roc. {I}nternat.
  {S}ummer {S}chool, {U}niv. {A}ntwerp, {A}ntwerp, 1972)},
      series={Lecture Notes in Math.},
      volume={Vol. 350},
   publisher={Springer, Berlin-New York},
       pages={1\ndash 55},
      review={\MR{406931}},
}

\bib{tsak}{incollection}{
      author={Tsaknias, Panagiotis},
       title={A possible generalization of {M}aeda's conjecture},
        date={2014},
   booktitle={Computations with modular forms},
      series={Contrib. Math. Comput. Sci.},
      volume={6},
   publisher={Springer, Cham},
       pages={317\ndash 329},
         url={https://doi.org/10.1007/978-3-319-03847-6_12},
      review={\MR{3381458}},
}

\bib{yoo}{article}{
      author={Yoo, H.},
       title={Non-optimal levels of a reducible {$\rm{mod}\,\ell$} modular
  representation},
        date={2019},
        ISSN={0002-9947,1088-6850},
     journal={Trans. Amer. Math. Soc.},
      volume={371},
      number={6},
       pages={3805\ndash 3830},
         url={https://doi.org/10.1090/tran/7314},
      review={\MR{3917209}},
}

\end{biblist}
\end{bibdiv}
	\bibliographystyle{alpha}
	\end{document}